\documentclass [11pt,oneside] {amsart}
\usepackage[leqno]{amsmath}
\usepackage{amsthm}
\usepackage{dsfont}
\usepackage[euler]{textgreek}
\usepackage{amsfonts}
\usepackage{amssymb}
\usepackage{eucal}
\usepackage{upgreek}
\usepackage{latexsym,amscd,fullpage, tikz, cmll, mathrsfs, epic, mathtools} 
\usepackage{extpfeil}
\usepackage{color}
\usepackage{graphicx}
\usepackage [all]{xy}
\usepackage{hyperref}
\usepackage{stmaryrd}

\usetikzlibrary{fit}
\usetikzlibrary{arrows}
\usetikzlibrary{matrix}
\SelectTips{cm}{12}

\theoremstyle{definition}
  \newtheorem{theorem}{Theorem}[subsection]
  \newtheorem{proposition}[theorem]{Proposition}
  \newtheorem{lemma}[subsubsection]{Lemma}
  \newtheorem{corollary}[theorem]{Corollary}

\theoremstyle{remark} 
  \newtheorem{remark}[subsubsection]{Remark}
  \newtheorem{example}[subsubsection]{Example}

\theoremstyle{definition}
  \newtheorem{definition}[theorem]{Definition}

\renewcommand {\geq} {\geqslant}

\newcommand {\tensor}{\otimes}
\newcommand{\Spec}[1]{\mathrm{Spec}\ \!#1}
\newcommand{\ZZ}{\mathbb{Z}}

\newcommand{\betterhookarrow}{\!\xymatrix{{}\ar@{^{(}->}[r]&{}}\! }

\def\1{{\bf 1}}
\def\lra{\longrightarrow}

\newcommand{\an}[1]{{#1}^{\mathrm{an}}}
\newcommand{\ad}[1]{{#1}^{\mathrm{ad}}}

\newcommand{\Trop}[2]{\mathrm{Trop}(#1,#2)}

\newcommand{\RR}{\mathbb{R}}
\newcommand{\TT}{\mathbb{T}}

\newextarrow{\xbigtoto}{{20}{20}{20}{20}}
   {\bigRelbar\bigRelbar{\bigtwoarrowsleft\rightarrow\rightarrow}}

\DeclareSymbolFont{largesym}{OML}{cmm}{m}{it}
\DeclareMathSymbol{\nstnsmall}{0}{largesym}{"22}

\makeatletter
\def\@tocline#1#2#3#4#5#6#7{\relax
  \ifnum #1>\c@tocdepth 
  \else
    \par \addpenalty\@secpenalty\addvspace{#2}%
    \begingroup \hyphenpenalty\@M
    \@ifempty{#4}{%
      \@tempdima\csname r@tocindent\number#1\endcsname\relax
    }{%
      \@tempdima#4\relax
    }%
    \parindent\z@ \leftskip#3\relax \advance\leftskip\@tempdima\relax
    \rightskip\@pnumwidth plus4em \parfillskip-\@pnumwidth
    #5\leavevmode\hskip-\@tempdima
      \ifcase #1
       \or\or \hskip 1em \or \hskip 2em \else \hskip 3em \fi%
      #6\nobreak\relax
    \dotfill\hbox to\@pnumwidth{\@tocpagenum{#7}}\par
    \nobreak
    \endgroup
  \fi}
\makeatother
\usepackage{hyperref}

\newcommand{\mono}{\!\xymatrix{{}\!\ar@{^{(}->}[r]&\!{}}\!}

\title{Introduction to adic tropicalization}
\author{Tyler Foster}
\address{{\bf Tyler Foster}\newline Department of Mathematics\\
University of Michigan\\
Ann Arbor, MI 48109, USA}
\email{tyfoster@umich.edu}
\date {\today}

\begin{document}
\maketitle

\begin{abstract}
\vskip -.25cm
This is an expository article on the adic tropicalization of algebraic varieties. We outline joint work with Sam Payne in which we put a topology and structure sheaf of local topological rings on the exploded tropicalization. The resulting object, which blends polyhedral data of the tropicalization with algebraic data of the associated initial degenerations, is called the {\em adic tropicalization}. It satisfies a theorem of the form ``Huber analytification is the limit of all adic tropicalizations." We explain this limit theorem in the present article, and illustrate connections between adic tropicalization and the curve complexes of O. Amini and M. Baker.
\end{abstract}

\tableofcontents


\section{Introduction}\label{introduction}
	
	
	
	There is a close relationship between tropical geometry and the geometry of degenerations of algebraic varieties. This relationship takes a particularly suggestive form in much of the recent work that reinterprets mirror symmetry and enumerative geometry in terms of fibrations of complex varieties over affine manifolds. \cite{KS1} \cite{Gross-Siebert} \cite{KS2} \cite{Gross}. An interesting manifestation of these ideas appears in the recent work of B. Parker \cite{Parker1} \cite{Parker2} \cite{Parker3}. Working in the paradigm of symplectic manifolds, with a view toward applications to pseudoholomorphic curve counting, Parker defines topological spaces called {\em exploded fibrations}, which he uses to construct log Gromov-Witten invariants \cite{ParkerLog1}, as required by the Gross-Siebert program. Roughly speaking, an exploded fibration is a topological space with structure semiring, which can be glued along boundary strata from pieces of the form $\big(\mathbb{C}^{\times}\big)^{n}\times P$, where $P$ is an integral polytope in $\RR^{m}$, for some $m$. For further details, see \cite[\S4]{Parker1} and \cite[\S3]{Parker2}. An illustrative example of an exploded fibration is the {\em exploded curve} given by the fibration $p:Y\lra\!\!\!\!\rightarrow B$ of topological spaces where $B\subset\RR^{2}$ is the union of the subspaces
	$$
	\scalebox{.95}{$
	B_{1}=\big\{(x,0)\in\RR^{2}:x\ge0\big\},
	\ \ \ \ 
	B_{2}=\big\{(0,y)\in\RR^{2}:y\ge0\big\},
	\ \ \ \ \mbox{and}\ \ \ \ 
	B_{3}=\big\{(x,x)\in\RR^{2}:x\le0\big\},
	$}
	$$
and where the space $Y$ is glued from pieces
	$$
	B_{1}\times\mathbb{C}^{\times}-\big((0,0),\ \!0\big),
	\ \ \ \ \ \ \ 
	B_{2}\times\mathbb{C}^{\times}-\big((0,0),\ \!1\big),
	\ \ \ \ \ \ \ \mbox{and}\ \ \ \ \ \ \ 
	B_{3}\times\mathbb{C}^{\times}-\big((0,0),\ \!\infty\big)
	$$
along an inclusion of $\big\{(0,0)\big\}\times\big(\mathbb{P}^{1}_{\mathbb{C}}-\{0,1,\infty\}\big)$ into each $B_{i}\times\mathbb{C}^{\times}$, $i=1,2,3$.
	\begin{figure}[!htb]
	$$
	\includegraphics[scale=.25]{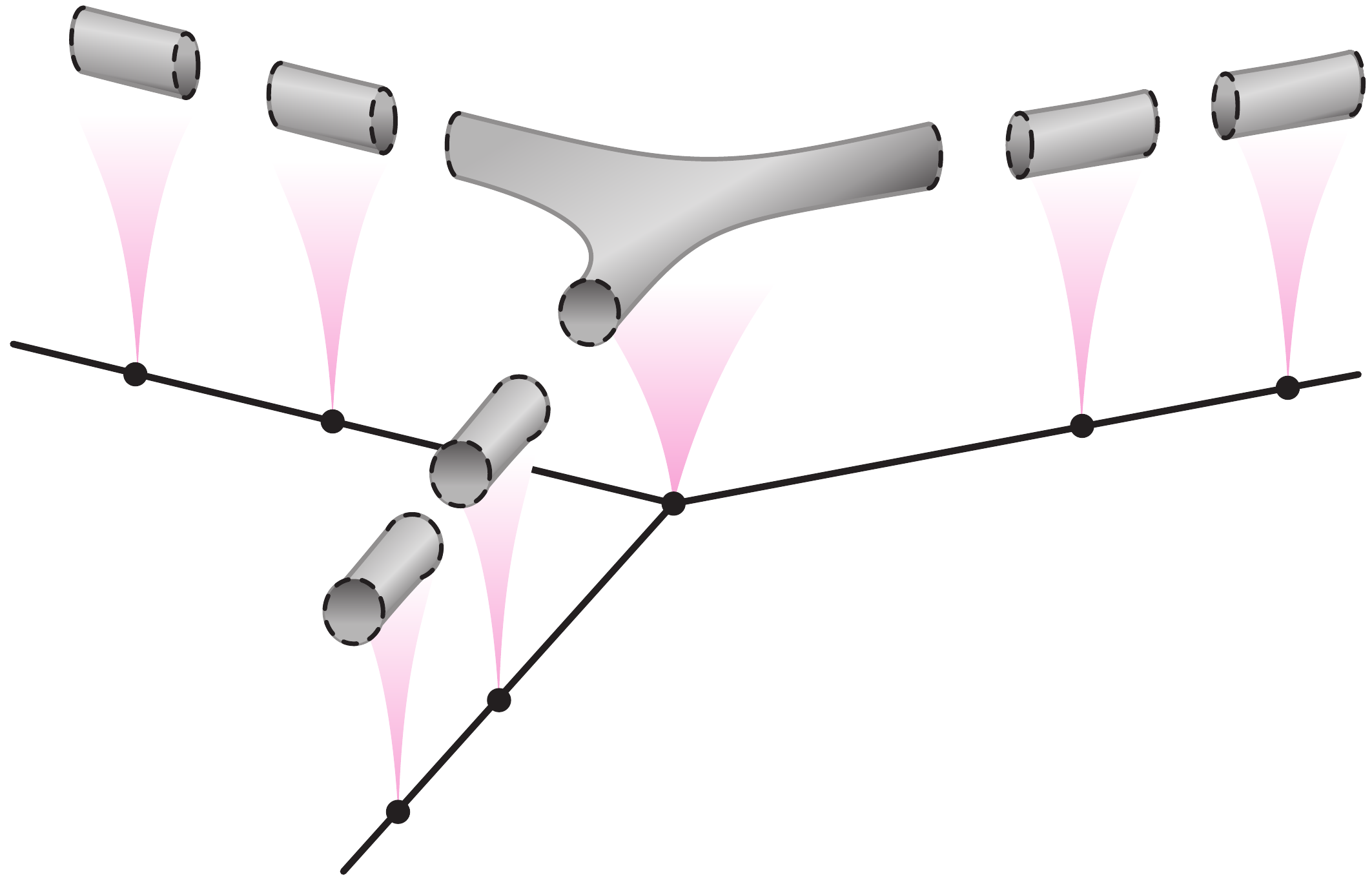}
	$$
	\caption{An example of an {\em exploded curve} in the sense of B. Parker \cite{Parker2}.}
	\label{figure: exploded curve}
	\end{figure}

	A tropical geometer recognizes the exploded curve above as a close analogue to the family of initial degenerations over all rational points in the tropicalization of an algebraic $K$-curve $X$. In more detail, let $K$ be a non-Archimedean field, which is to say that $K$ is complete with respect to a non-Archimedean absolute value $|-|:K^{\times}\lra\RR_{\ge0}$. Assume that $K$ is algebraically closed, with value group $\Gamma\overset{{}_{\mathrm{def}}}{=}-\mathrm{log}_{\ \!}|K^{\times}|\subset\mathbb{R}$. Consider an $n$-dimensional split algebraic torus $\mathbb{T}$ with character lattice $M$ and cocharacter lattice $N=\mathrm{Hom}_{\ZZ}(M,\ZZ)$. Each vector $v\in N_{\mathbb{R}}\overset{{}_{\mathrm{def}}}{=}\mathbb{R}\tensor_{\ZZ}N$ determines the {\em tilted group ring} $R[U_{v}]$, the ring generated by all those monomials $a\chi^{u}$ for which $\langle u,v\rangle-\mathrm{log}_{\ \!}|a|\geq0$. Its spectrum $T_{v}\overset{{}_{\mathrm{def}}}{=}\mathrm{Spec}_{\ \!}R[U_{v}]$ is an $R$-scheme with generic fiber $T_{v,K}\cong\mathbb{T}$. When $v\in N_{\Gamma}\overset{{}_{\mathrm{def}}}{=}\Gamma\tensor_{\ZZ}N$, this ring $R[U_{v}]$ is finitely generated, with special fiber of the same dimension as its generic fiber $\mathbb{T}$. If $X$ is any $K$-scheme, realized as a closed subvariety of the $K$-torus $\mathbb{T}$ via some closed embedding $\textit{\i}:\xymatrix{X\ar@{^{(}->}[r]&\mathbb{T}}$, then the {\em initial degeneration} of $X$ at $v$, denoted $\mathrm{in}_{v}X$, is the special fiber $(\overline{X})_{k}$ of the closure $\overline{X}$ inside $T_{v}$.
	
	The Fundamental Theorem of Tropical Geometry \cite[Theorem 3.2.4]{IntroTrop} says that the tropicalization of $X$ is the closure, in $N_{\mathbb{R}}$, of the set of those vectors $v\in N_{\Gamma}$ at which the initial degeneration $\mathrm{in}_{v}X$ is nonempty:
$$
\mathrm{Trop}(X,\textit{\i})
\ =\ 
\overline{\big\{v\in N_{\Gamma}\ \!:\ \!\mathrm{in}_{v}X\ne \emptyset \}}.
$$
This means that we can retain more information about the subvariety $X\subset\mathbb{T}$ by remembering not just the points of the tropical variety $\mathrm{Trop}(X,\textit{\i})$, but the points of all initial degenerations of $X$ at all vectors $v\in N_{\Gamma}$. Payne makes use of this set-theoretical union
	$$
	\bigsqcup_{v\in\mathrm{Trop}(X)\cap N_{\Gamma}}\!\!\!\!\!\!\!\!|\mathrm{in}_{v}X|
	$$
in order to study the fibers of the tropicalization map $X(K)\lra\mathrm{Trop}(X)$ \cite{Fiber} \cite{Fiber2}. He refers to this set as the ``exploded tropicalization," taking inspiration from Parker's exploded fibrations. In the present paper, we keep track of initial degenerations at all points of $\mathrm{Trop}(X)$, $\Gamma$-rational or not, and define the {\em exploded tropicalization} to be the union
	$$
	\mathfrak{Trop}(X)
	\ \ \overset{\mathrm{def}}{=}\!
	\bigsqcup_{v\in\mathrm{Trop}(X)}\!\!\!\!\!\!|\mathrm{in}_{v}X|.
	$$
	
	One major difference between Parker and Payne's exploded objects is the presence of both a topology {\em and} a structure sheaf (of semirings) on Parker's exploded fibrations, whereas Payne's exploded tropicalizations are mere sets. The present paper is an exposition that complements the forthcoming research paper \cite{FP}, joint with Sam Payne, in which we explain how to put both a topological structure and a structure sheaf (of topological rings) on the exploded tropicalization $\mathfrak{Trop}(X)$ of any closed subvariety $X$ of a torus $\TT$ over a non-Archimedean field $K$. Our construction turns $\mathfrak{Trop}(X)$ into a locally topologically ringed space. We denote this space $\mathrm{Ad}(X)$, and refer to it as the {\em adic tropicalization} of $X$ in $\TT$.
	
	The locally topologically ringed space $\mathrm{Ad}(X)$ interacts with Huber's adic spaces \cite{HubGen} \cite{Huber} in much the same way that tropical varieties interact with Berkovich spaces. In \cite{P09}, Payne shows that the if $X$ is a quasi-projective variety over an algebraically closed non-Archimedean field $K$, then the topological space underlying the Berkovich analytification $X^{\mathrm{an}}$ is the inverse limit of all the tropical varieties $\mathrm{Trop}(X,\textit{\i})$, taken over the inverse system formed by all closed embeddings $\textit{\i}:X\mono Y_{\Sigma}$ into quasi-projective toric varieties $Y_{\Sigma}$. In \cite{FGP}, P. Gross, S. Payne, and the author give criterion extending this result to the case of certain non-quasi-projective varieties $X$ over an algebraically closed, non-trivially valued non-Archimedean field $K$. In particular, \cite[Theorem 1.2]{FGP} $X$ says that if $X$ is a closed subvariety of a toric $K$-variety, then the topological space underlying $X^{\mathrm{an}}$ is the inverse limit of tropical varieties $\mathrm{Trop}(X,\textit{\i})$, taken over the inverse system formed by all closed embeddings $\textit{\i}:X\mono Y_{\Sigma}$ into arbitrary toric varieties $Y_{\Sigma}$. One of the central results of the forthcoming paper \cite{FP} is a theorem stating that the Huber analytification $X^{\mathrm{ad}}$ of a closed subvariety $X$ of a proper toric variety is isomorphic, as a locally topologically ringed space, to the inverse limit of the adic tropicalizations of $X$ associated to all closed embeddings of $X$ into proper toric varieties. Section \ref{Section-Main Theorem} of the present paper provides a brief exposition on this result. See Theorem \ref{Huberlimit} and Proposition \ref{Nice system, man!} below for precise statements.
	
	Because the present paper is expository, it omits most proofs. The interested reader can find all the proofs in \cite{FP}. The present paper makes up for its lack of proofs by providing concrete details and examples that do not appear in \cite{FP}. In what follows, we focus more heavily on the basic geometry of adic spaces, which underlies our geometrization of $\mathfrak{Trop}(X)$, and on concrete examples of the resulting topological spaces and their relationship to the metrized curve complexes of O. Amini and M. Baker \cite{Amini-Baker}.

\vskip .3cm

\noindent \textbf{Acknowledgments.}  The author would like to thank the Simons Foundation and the organizers of the Simons Symposium on non-Archimedean and Tropical Geometry for bringing together a fantastic group of researchers for an incredibly stimulating week of mathematics. Special thanks go to S. Payne, as this paper is a report on joint with him. The author thanks M. Baker, F. Baldassarri, W. Gubler, and D. Ranganathan for helpful remarks. Finally, the author thanks the reviewer for comments that greatly improved this paper.

The author is supported by NSF RTG grant DMS-0943832.


\vskip .75cm

\section{Preliminaries on Huber analytification}\label{section: preliminaries on Huber analytification}

	In the present \S\ref{section: preliminaries on Huber analytification}, we review the basics of Huber's theory of adic spaces and Huber's analytification functor, which produces an adic space from any $K$-scheme $X$. The original sources for this material are \cite{HubGen} and \cite{Huber}. We also strongly recommend T. Wedhorn's notes \cite{Wed} for a rather accessible introduction.
	
	The reader already well versed in the theory of adic spaces can skip ahead to \S\ref{Section-adic trop}.
	
\vskip .3cm

\noindent
{\bf Notation.}
Fix a {\em non-Archimedean} field $K$, i.e., a field complete with respect to some non-trivial, non-Archimedean valuation $v:K\lra\mathbb{R}\sqcup\{\infty\}$.
Assume that $K$ is algebraically closed. Let $R$ denote the ring of integers in $K$, let $\mathfrak{m}$ denote the unique maximal ideal in $R$, and let $k=R/\mathfrak{m}$ denote the residue field. Choose a real number $0<\varepsilon<1$ once and for all, and let
$|-|_{v}:K\lra\mathbb{R}_{\geq0}$
denote the norm $|a|_{v}=\varepsilon^{\mathrm{val}(a)}$ associated to $v$.
	
	Recall that the {\em Tate algebra}, denoted $K\langle T_{1},\dots,T_{n}\rangle$, is the subalgebra of $K[[T_{1},\dots,T_{n}]]$ consisting of those formal power series $f=\sum a_{i_{1},\dots,i_{n}}T^{m_{i_{1}}}_{i_1}\cdots T^{m_{i_{n}}}_{i_n}$ for which $|a_{i_{1},\dots,i_{n}}|_{v}\rightarrow0$ as $i_{1}+\cdots+i_{n}\rightarrow\infty$. The Tate algebra comes with the {\em Gauss\ norm} $\|-\|_{\mathrm{G}}:K\langle T_{1},\dots,T_{n}\rangle\lra\mathbb{R}_{\ge0}$, which takes $f\mapsto\|f\|_{\mathrm{G}}\overset{{}_{\mathrm{def}}}{=}\mathrm{max}|a_{i_{1},\dots,i_{n}}|_{v}$.
The Gauss norm induces a {\em quotient norm} $\|-\|_{\mathrm{q}}$ on any quotient $\mathrm{pr}:\xymatrix{K\langle T_{1},\dots,T_{n}\rangle\ar@{->>}[r]&K\langle T_{1},\dots,T_{n}\rangle/\mathfrak{a}}$, defined according to
$$
\|f\|_{\mathrm{q}}
\ \ \overset{\mathrm{def}}{=}\ \ 
\mathrm{inf}\big\{\|g\|_{\mathrm{G}}\ \!:\ \!g\in\mathrm{pr}^{-1}(f)\big\}.
$$
This quotient norm $\|-\|_{\mathrm{q}}$ gives $K\langle T_{1},\dots,T_{n}\rangle/\mathfrak{a}$ the structure of a commutative $K$-Bannach algebra. A $K$-{\em affinoid algebra} is any commutative Banach algebra $A$ admitting at least one isomorphism of $K$-Banach algebras $A\cong K\langle T_{1},\dots,T_{n}\rangle/\mathfrak{a}$. We use ``$\|-\|_{\mathrm{q}}$" to denote the norm on an arbitrary affinoid algebra $A$, even when there is no explicit mention of the presentation of $A$ that gives rise to $\|-\|_{\mathrm{q}}$.

\vskip .3cm


\subsection{Adic spectra and adic spaces}
	For an arbitrary totally ordered abelian group $\Gamma'$, written multiplicatively, let $\{0\}\sqcup\Gamma'$ denote the totally ordered abelian semigroup where $0<\gamma$ for all $\gamma\in\Gamma'$, and where $0\cdot\gamma=0$ for all $\gamma\in\{0\}\sqcup\Gamma'$. A {\em continuous seminorm} $\|-\|_{x}:A\lra\{0\}\sqcup\Gamma'$ is any homomorphism of multiplicative semigroups satisfying
$$
\|0\|_{x}=0,\ \ \ \ \ \ \|1\|_{x}=1,\ \ \ \ \ \ \mbox{and}\ \ \ \ \ \ \|a+b\|_{x}\leq\mathrm{max}(\|a\|_{x},\|b\|_{x}),
$$
such that the set
	$
	\{a\in A:\|a\|_{x}<\gamma\}
	$
is open for each $\gamma\in\Gamma'$. Two continuous seminorms
	$$
	\|-\|_{x}:A\lra\{0\}\sqcup\Gamma_{1}
	\ \ \ \ \ \ \ \mathrm{and}\ \ \ \ \ \ \ 
	\|-\|_{y}:A\lra\{0\}\sqcup\Gamma_{2}
	$$
are {\em equivalent} if there exists an inclusion $\alpha:\{0\}\sqcup\Gamma_{1}\betterhookarrow \{0\}\sqcup\Gamma_{2}$ of ordered abelian semigroups such that $\alpha{}_{{}^{\ \!\circ}}\|-\|_{x}=\|-\|_{y}$.

Let $A$ be a $K$-affinoid algebra. The norm $\|-\|_{\mathrm{q}}$ on $A$ determines the {\em power bounded subring} $A^{\circ}\overset{{}_{\mathrm{def}}}{=}\big\{f\in A:\|f\|_{\mathrm{q}}\leq1\big\}$. The {\em adic spectrum} of the pair $(A,A^{\circ})$, denoted $\mathrm{Spa}(A,A^{\circ})$, is the set of all equivalence classes of continuous seminorms $\|-\|_{x}:A\lra\{0\}\sqcup\Gamma$ such that $\|a\|_{x}\leq1$ for all $a\in A^{\circ}$. The topology on $\mathrm{Spa}(A,A^{\circ})$ is the topology generated by its {\em rational subsets}, that is, the coarsest topology containing the images of all inclusions
	$$
	\xymatrix{\mathrm{Spa}\Big(A\big\langle\!\frac{f_{1},\dots,f_{m}}{g}\!\big\rangle,\ A\big\langle\!\frac{f_{1},\dots,f_{m}}{g}\!\big\rangle^{\!\circ\ }\Big)\ \ 
\ar@{^{(}->}[r]
&
\ \ \mathrm{Spa}(A,A^{\circ})},
	$$
where $A\Big\langle\!\frac{f_{1},\dots,f_{m}}{g}\!\Big\rangle$ is any rational $K$-algebra determined by a set of elements $\{f_{1},\dots,f_{m},g\}\subset A$ generating the unit ideal in $A$ (see \cite[\S2.2.2.(ii)]{Berk} and \cite[\S8.1]{Wed}).

If $A$ is an affinoid algebra, with adic spectrum $X=\mathrm{Spa}(A,A^{\circ})$, and if $U$ is a rational subset of $X$, let $\mathscr{O}_{X}(U)$ denote a choice of rational affinoid algebra associated to $U$ as above.
These rational algebras induce a {\em structure presheaf} $\mathscr{O}_{X}$ on the adic spectrum $X$, which returns, at each open subset $U'\subset X$, the topological ring $\mathscr{O}_{X}(U')$ given by the inverse limit
$$
\mathscr{O}_{X}(U')
\ \ =\ \ 
\varprojlim_{U\subset U'} \mathscr{O}_{X}(U)
$$
taken over all rational subsets $U$ of $U'$ in $X$. The theorem \cite[Theorem 2.2]{HubGen} implies that the resulting structure presheaf $\mathscr{O}_{\!X}$ is in fact a sheaf. Furthermore, we can form the stalk $\mathscr{O}_{X,x}$ at each point $x$ in $\mathrm{Spa}(A,A^{\circ})$. This stalk is a local topological ring, and the seminorm $\|-\|_{x}$ induces a continuous seminorm on $\mathscr{O}_{X,x}$ \cite[Proposition 1.6]{HubGen}.

\subsubsection{{\bf Adic spaces.}}
	Let $\bold{LRS}$ donote the category of locally ringed spaces. By a {\em locally topologically ringed space}, we mean a locally ringed space $(Y,\mathscr{O}_{Y})$ such that $\mathscr{O}_{Y}$ comes with the structure of a sheaf of topological rings. Let $\bold{LTRS}$ denote the category of locally topologically ringed spaces.

	An {\em adic space} over $K$ is a locally topologically ringed space $(X,\mathscr{O}_{X})$ that comes equipped with a continuous seminorm $\|-\|_{x}$ on the stalk $\mathscr{O}_{X,x}$ at each point $x\in X$, and which admits an atlas $\{U_{i}\betterhookarrow X\}$ in $\bold{LTRS}$, whose charts $U_{i}$ are adic spectra of $K$-affinoid algebras, such that the seminorm $\|-\|_{x}$ on each stalk $\mathscr{O}_{X,x}=\mathscr{O}_{U_{i},x}$ coincides with the seminorm described in the previous paragraph. Note that this means, in particular, that the structure presheaf $\mathscr{O}_{X}$ on any adic space must be a sheaf. A {\em morphism} of adic spaces is any morphism $\varphi:(X,\mathscr{O}_{X})\lra(Y,\mathscr{O}_{Y})$ of locally topologically ringed spaces such that, at each point $x\in X$, the induced morphism
$$
\varphi^{\#}:\mathscr{O}_{Y,\varphi(x)}\lra\mathscr{O}_{X,x}
$$
on stalks satisfies $\|-\|_{x}{}_{{}^{\ \!\circ}}\varphi^{\#}=\|-\|_{\varphi(x)}$. We let $\bold{Adic}_{K}$ denote the category of adic spaces over $K$.

\vskip .3cm

\subsubsection{{\bf The sheaf of power bounded sections.}} 
Every adic space $(X,\mathscr{O}_{X})$ comes with a second sheaf of topological rings, called the {\em sheaf of power bounded sections}. This is the subsheaf $\mathscr{O}^{\ \!\circ}_{\!X}\subset\mathscr{O}_{X}$ that, on each open subset $U\subset X$, returns those sections $f\in\mathscr{O}_{X}(U)$ for which $\|f\|_{x}\leq1$ at every point $x\in U$.

	The sheaf of power bounded sections is closely related to the theory of degenerations of $K$-varieties to the special fiber in $\mathrm{Spec}_{\ \!}R$, and it will play an important role in our formulation of the statement that ``Huber analytification is the limit of all exploded tropicalizations" (see Theorem \ref{Huberlimit} below).

\vskip .75cm


\subsection{Huber analytification}
Fix a separated scheme $X$ of finite type over $K$. It is a space defined locally by the vanishing of algebraic functions, and has no analytic structure. The Berkovich analytification $\an{X}$ provides one way to realize an analytic structure on $X$. There is also an adic space $\ad{X}$ that we can associate to $X$, called the Huber analytification of $X$. Intuitively, the Huber analytification of $X$ is the pullback of $X$ to the category of adic spaces over $K$ along the morphism $\mathrm{Spa}(K,R)\lra\Spec{K}$ of locally ringed spaces.

To make this precise, observe that there is a forgetful functor $\mathrm{U}:\bold{Adic}_{K}\lra\bold{LRS}$ that takes a given adic space $Y$ and forgets both the seminorms on its stalks $\mathscr{O}_{Y,y}$ and the topological structure on the sheaf $\mathscr{O}_{Y}$. Let $\bold{LRS}_{K}$ denote the category of locally ringed spaces over $\Spec{K}$. There is a canonical isomorphism
	\begin{equation}\label{equation: forgetful functor}
	\mathrm{U}\big(\mathrm{Spa}(K,R)\big)\xrightarrow{\ \sim\ }\Spec{K}
	\end{equation}
of locally ringed spaces, which lets us interpret the forgetful functor $\mathrm{U}$ as a functor of the form
$$
\mathrm{U}:\bold{Adic}_{K}\lra\bold{LRS}_{K}.
$$

\vskip .3cm

\begin{definition}\label{adification}
{\bf (Huber analytification)}. If $X$ is a separated scheme of finite type over $K$, then its {\em Huber analytification}, denoted $\ad{X}$, is any adic space over $\mathrm{Spa}(K,R)$ that comes equipped with a morphism $\mathrm{U}(\ad{X})\lra X$ of locally ringed spaces over $\Spec{K}$ satisfying the following universal property:

\begin{itemize}
\item[{\bf (H)}]
If $Y$ is any adic space over $K$ that comes equipped with a map
$\varphi:\mathrm{U}(Y)\lra X$ of locally ringed spaces over $\Spec{K}$, then there exists a unique morphism $\widetilde{\varphi}:Y\lra\ad{X}$ of adic spaces over $K$ for which the diagram
$$
\ \ \ \ \ \ 
\begin{xy}
(0,0)*+{X}="1";
(-17,10)*+{\mathrm{U}(Y)}="2";
(-17,-10)*+{\mathrm{U}(\ad{X})}="3";
{\ar@{->}^{\varphi} "2"; "1"};
{\ar@{->} "3"; "1"};
{\ar@{-->}_{\mathrm{U}({\widetilde{\varphi}})\ } "2"; "3"};
\end{xy}
$$
commutes.
\end{itemize}
\end{definition}

\vskip .2cm

\begin{remark}
The universal property {\bf (H)} gives the Huber analytification the structure of a functor
	$$
	\ad{(-)}:\bold{Sch}_{K}\lra \bold{Adic}_{K}.
	$$
	
	For an explicit construction of the adic space $\ad{X}$ as a fiber product of locally ringed spaces, see \cite[\S 8.7 and Definition 8.63]{Wed}. Most relevant to our purposes is the alternate description of $\ad{X}$ as an inverse limit over admissible formal models of $\an{X}$, as established by M. van der Put and P. Schneider in \cite{Points}. To this end, we briefly review the theory of admissible formal models of $X$.
\end{remark}

\vskip .3cm

\subsubsection{{\bf Admissible formal models.}}
Recall that $\mathfrak{m}$ denotes the maximal ideal in $R$, and that $k=R/\mathfrak{m}$ denotes our residue field. The theory of admissible formal models provides us with a tool for degenerating adic spaces over $K$ to algebraic varieties over $k$. These degenerations are realized as certain formal $R$-schemes. Not all formal $R$-schemes will do. For instance, the category of formal $R$-schemes includes $k$-schemes as a full subcategory. A $k$-scheme is an example of a formal $R$-scheme that contains no information whatsoever over the generic point of $\mathrm{Spec}_{\ \!}R$. We want to excise these from our category of interest.

We say that a topological $R$-algebra $A$ is {\em admissible} if there exists an isomorphism
$$
A
\ \ \cong\ \ 
R\langle t_{1},\dots,t_{n}\rangle\big/\mathfrak{a}
$$
of topological rings, where $R\langle t_{1},\dots,t_{n}\rangle\big/\mathfrak{a}$ has its $\mathfrak{m}$-adic topology, such that:
\vskip .2cm
\begin{itemize}
\item[{\bf (i)}]
the ideal $\mathfrak{a}\subset R\langle t_{1},\dots,t_{n}\rangle$ is finitely generated;
\item[{\bf (ii)}]
\vskip .2cm
the ring $R\langle t_{1},\dots,t_{n}\rangle\big/\mathfrak{a}$ is free of $\mathfrak{m}$-torsion.
\end{itemize}
\vskip .2cm
An {\em admissible formal $R$-scheme} is any formal scheme $\mathfrak{X}$ over the formal spectrum $\mathrm{Spf}_{\ \!\!}R$ of $R$ (see \cite[Chp. II, \S 9]{Hart} for details) that admits an open covering
$\big\{\mathfrak{U}_{i}\xymatrix{{}\ar@{^{(}->}[r]&{}}\mathfrak{X}\big\}$ by formal spectra $\mathfrak{U}_{i}=\mathrm{Spf}A_{i}$, such that each $A_{i}$ is an admissible $R$-algebra.

Each admissible formal $R$-scheme $\mathfrak{X}$ has an associated adic space $(\ad{\mathfrak{X}},\mathscr{O}_{\ad{\mathfrak{X}}})$.
To construct $\ad{\mathfrak{X}}$, choose a formal affine cover
$\big\{\mathfrak{U}_{i}\betterhookarrow \mathfrak{X}\big\}$,
say with $\mathfrak{U}_{i}=\mathrm{Spf}\ \!A_{i}$, where $A_{i}$ is an admissible $R$-algbera. Then $K\tensor_{R}A_{i}$ is a $K$-affinoid algebra, and we can define
$$
\ad{\mathfrak{U}}_{i}
\ \ \overset{\mathrm{def}}{=}\ \ 
\mathrm{Spa}\big(K\tensor_{R}A_{i}\ \!,\ (K\tensor_{R}A_{i})^{\circ}\big).
$$
These adic spaces glue along the intersections
$\ad{(\mathfrak{U}_{i}\cap\mathfrak{U}_{j})}=\ad{\mathfrak{U}}_{i}\cap\ad{\mathfrak{U}}_{j}$ to produce an adic space over $K$ that we denote $\ad{\mathfrak{X}}$. The adic space that we obtain in this way is independent of our choice of covering $\big\{\mathfrak{U}_{i}\betterhookarrow \mathfrak{X}\big\}$ (see \cite[Proposition 4.1]{HubGen} for details).

\vskip .3cm

\begin{definition}\label{admissible model of a scheme}
{\bf (Admissible formal models of a scheme).}
Let $X$ be any scheme locally of finite type over $K$. An {\em admissible formal model of} $X$ is a datum consisting of an admissible formal $R$-scheme $\mathfrak{X}$ and an isomorphism $\ad{\mathfrak{X}}\xrightarrow{\ \sim\ }\ad{X}$ of adic spaces.

A {\em morphism} of admissible formal models of $X$ is a morphism $\mathfrak{X}_{1}\lra\mathfrak{X}_{2}$ whose induced morphism $\mathfrak{X}_{1}^{\mathrm{ad}}\lra\mathfrak{X}_{2}^{\mathrm{ad}}$ of $K$-analytic spaces commutes with the isomorphisms $\mathfrak{X}_{1}^{\mathrm{ad}}\xrightarrow{\ \sim\ }\ad{X}$ and $\mathfrak{X}_{2}^{\mathrm{ad}}\xrightarrow{\ \sim\ }\ad{X}$. We let $\bold{AFS}_{R}$ denote the category of admissible formal $R$-schemes, and we let $\bold{AFS}_{R}^{X}$ denote the category of admissible formal models of $X$.
\end{definition}

\vskip .3cm

\subsubsection{{\bf Specializations of adic spaces.}}\label{specializations of adic spaces}
	The adic space $\ad{\mathfrak{X}}$ associated to the formal $R$-scheme
$\mathfrak{X}$ comes with a {\em specialization morphism}
\begin{equation}\label{formal specialization}
\mathrm{sp}_{\mathfrak{X}}:(\ad{\mathfrak{X}},\mathscr{O}^{\ \!\circ}_{\!\ad{\mathfrak{X}}})\lra(\mathfrak{X},\mathscr{O}_{\mathfrak{X}})
\end{equation}
of locally topologically ringed spaces over $\mathrm{Spec}_{\ \!}K$. Recall that $\mathscr{O}^{\ \!\circ}_{\!\ad{\mathfrak{X}}}$ is the sheaf of power bounded sections on $\ad{\mathfrak{X}}$. The fact that $\mathrm{sp}_{\mathfrak{X}}$ is a morphism of locally ringed spaces is one feature of adic spaces that Berkovich spaces do {\em not} satisfy, as the reduction map $X^{\mathrm{an}}\lra\mathfrak{X}_{k}$ for Berkovich spaces fails to be continuous.

	The specialization morphism (\ref{formal specialization}) satisfies the following universal property:
\vskip .3cm
\begin{itemize}
\item[{\bf (H${}_{\bold{f}}$)}]
If $Y$ is any adic space over $K$ that comes equipped with a morphism $\varphi:(Y,\mathscr{O}^{\ \!\!\circ}_{\!Y})\lra(\mathfrak{X},\mathscr{O}_{\mathfrak{X}})$ of locally topologically ringed spaces over $\mathrm{Spec}_{\ \!}K$, then there exists a unique morphism $\widetilde{\varphi}:(Y,\mathscr{O}_{Y})\lra(\ad{\mathfrak{X}},\mathscr{O}_{\ad{\mathfrak{X}}})$ of adic spaces over $K$ for which the diagram
$$
\ \ \ \ \ \ 
\begin{xy}
(0,0)*+{(\mathfrak{X},\mathscr{O}_{\mathfrak{X}})}="1";
(-25,10)*+{(Y,\mathscr{O}^{\ \!\!\circ}_{\!Y})}="2";
(-25,-10)*+{(\ad{\mathfrak{X}},\mathscr{O}^{\ \!\!\circ}_{\!\ad{\mathfrak{X}}})}="3";
{\ar@{->}^{\varphi} "2"; "1"};
{\ar@{->}_{\mathrm{sp}_{\mathfrak{X}}} "3"; "1"};
{\ar@{-->}_{\widetilde{\varphi}\ } "2"; "3"};
\end{xy}
$$
commutes.
\end{itemize}
\vskip .3cm

\begin{remark}\label{remark: Gillam}
The universal property {\bf (H${}_{\bold{f}}$)} gives us a functor
$\ad{(-)}:\bold{AFS}_{R}\lra\bold{Adic}_{K}$
into the category of adic spaces over $K$. If $X$ is a separated finite-type $K$-scheme, then the resulting specialization morphisms
\begin{equation}\label{adic specialization}
\mathrm{sp}_{\mathfrak{X}}:(\ad{X},\mathscr{O}^{\ \!\circ}_{\!\ad{X}})\lra(\mathfrak{X},\mathscr{O}_{\mathfrak{X}}),
\end{equation}
induced by the specialization morphisms (\ref{formal specialization}), commute with morphisms in $\bold{AFS}_{R}^{X}$.

	W. D. Gillam proved in \cite[Corollary 5]{Gillam} that $\bold{LRS}$ contains all inverse limits. It is not difficult to show that one can compute any inverse limit $\varprojlim_{}(Y_{i},\mathscr{O}_{Y_{i}})$ in $\bold{LTRS}$ by first computing the inverse limit $(Y,\mathscr{O}_{Y})$ in $\bold{LRS}$ of the underlying locally ringed spaces, and then equipping each ring $\mathscr{O}_{Y}(U)$ with the finest topology for which each of the morphisms 
	\begin{equation}\label{individual restrictions}
	\mathscr{O}_{Y_{i}}(V_{i})\lra\mathscr{O}_{X}(U),
	\end{equation}
for each open subset $V_{i}\subset Y_{i}$ mapping to $U$ under $Y_{i}\lra Y$,
becomes continuous.

This means, in particular, that the inverse limit of an inverse system of formal schemes remains a topologically ringed space on which stalks are local rings. Thus the morphisms (\ref{adic specialization}) induce a morphism
\begin{equation}\label{specialization}
\mathrm{sp}:\ \!(\ad{X},\mathscr{O}^{\ \!\!\circ}_{\!\ad{X}})\ \xrightarrow{\ \ \ \ \ }\varprojlim_{\bold{AFS}_{R}^{X}}\!(\mathfrak{X},\mathscr{O}_{\mathfrak{X}})
\end{equation}
of locally topologically ringed spaces.
\end{remark}

\vskip .3cm

\begin{proposition}\label{limit}
The morphism (\ref{specialization}) is an isomorphism of locally topologically ringed spaces.
\end{proposition}
\begin{proof}
For a sketch of the proof of Proposition \ref{limit}, see \cite[Theorem 2.22]{Sch}.
\end{proof}

\vskip .75cm


\section{The exploded tropicalization as a locally ringed space}\label{Section-adic trop}


\subsection{The combinatorics of toric degenerations}\label{subsection: The combinatorics of toric degenerations}
	In the present \S\ref{subsection: The combinatorics of toric degenerations}, we briefly review the combinatorial and algebraic geometry of toric degenerations of toric varieties. We work exclusively over the toric $k$-variety $\mathbb{A}^{1}_{k}$, where $k=R/\mathfrak{m}$ is our original residue field. The reader already familiar with the geometry of toric degenerations over $\mathbb{A}^{1}_{k}$, as developed by G. Kempf, F. F. Knudsen, D. Mumford, and B. Saint-Donat in \cite{KKMSD}, may proceed to \S\ref{subsection: algebraic Gubler models}, where we review the formalism of toric $R$-schemes introduced by W. Gubler in \cite{Gub}. The theory of toric $R$-schemes is further developed by W. Gubler and A. Soto in \cite{GS}.

	Fix a split, $n$-dimensional algebraic torus $\mathbb{T}$ over $k$, with character lattice $M$. We have a canonical isomorphism $\mathbb{T}\cong\mathrm{Spec}_{\ \!}k[M]$. Let $N\overset{{}_{\mathrm{def}}}{=}\mathrm{Hom}_{\ZZ}(M,\ZZ)$ be the cocharacter lattice, and let
	$$
	\langle-,-\rangle\ \!:\ M\times N_{\RR}\ \lra\ \RR
	$$
denote the canonical paring.
	
	If $Y_{\Sigma}$ is a proper toric variety over $K$, built from a complete $n$-dimensional fan $\Sigma$ in $N_{\RR}$, then we can freely produce degenerations of $Y_{\Sigma}$ by writing down certain $(n+1)$-dimensional fans extending $\Sigma$. Specifically, consider the $(n-1)$-dimensional half-space $N_{\RR}\times\RR_{\ge0}$, along with its projection to the second factor
	\begin{equation}\label{inset: projection to second factor}
	\mathrm{pr}_{2}\ \!:\ N_{\RR}\times\RR_{\ge0}\ \lra\!\!\!\!\rightarrow\ \RR_{\ge0}.
	\end{equation}
If $\Delta$ is any fan in $N_{\RR}\times\RR$ with support $N_{\RR}\times\RR_{\ge0}$, such that
	$$
	\Delta\ \cap\ \big(\ \!N_{\RR}\times\{0\}\ \!\big)
	\ \ \cong\ \ 
	\Sigma,
	$$
then the projection (\ref{inset: projection to second factor}) induces a morphism of fans
	\begin{equation}\label{inset: morphisms of fans}
	\mathrm{pr}_{2}\ \!:\ \Delta\ \lra\!\!\!\!\rightarrow\ \big\{0,\RR_{\ge0}\big\}
	\ \ \ \ \ \ \ \ \mbox{with}\ \ \ \ \ \ \ \ 
	\mathrm{pr}_{2}^{-1}(0)\ =\ \Sigma.
	\end{equation}
This induces a torus-equivariant morphism of toric varieties
	\begin{equation}\label{toric fibration}
	f_{\mathrm{pr}_{2}}\ \!:\ Y_{\Delta}\ \lra\!\!\!\!\rightarrow\ \mathbb{A}^{1}=\mathrm{Spec}_{\ \!}k[t].
	\end{equation}
See Figure \ref{toric degeneration} below for an example. Note that the fiber $\mathrm{pr}^{-1}_{2}(1)=\Delta\cap\big(N_{\RR}\times\{1\}\big)$ is not a fan but a polyhedral complex. We identify it with a polyhedral complex $C_{\Delta}$ in $N_{\RR}$ whose support is all of $N_{\RR}$. See Figure \ref{tropcone} below for an example.

\vskip .3cm
	
\begin{proposition}\label{proposition: how toric degenerations work}
	The map $f_{\mathrm{pr}_{2}}:Y_{\Delta}\lra\!\!\!\!\rightarrow\mathbb{A}^{1}$ induced by the projection (\ref{inset: morphisms of fans}) realizes the complex toric variety $Y_{\Delta}$ as an $\mathbb{A}^{1}$-model of the toric variety glued from $\Sigma$. More precisely:
\begin{itemize}
	\item[{\bf (i)}]\vskip .25cm
	The generic fiber of $f_{\mathrm{pr}_{2}}$ is the toric $k(t)$-variety $Y_{\Sigma}$.
	\item[{\bf (ii)}]\vskip .25cm
	The special fiber of $f_{\mathrm{pr}_{2}}$ at $t=0$ is a union of toric varieties $Y_{\mathrm{star}(v)}$ indexed by the vertices of $v$ of $C_{\Delta}$, glued along torus-invariant divisors according to the higher-dimensional face incidences appearing in $C_{\Delta}$.
\end{itemize}
\vskip .25cm
\end{proposition}
\begin{proof}
The complex toric variety $Y_{\Delta}$ is glued from the affine open toric varieties $U_{\delta}\overset{{}_{\mathrm{def}}}{=}\mathrm{Spec}_{\ \!}k[S_{\delta}]$, where $\delta$ runs over all cones in $\Delta$. Because $\Delta$ is a fan in $N_{\RR}\times\RR_{\ge0}$, the semigroup $S_{\delta}$ associated to $\delta$ takes the form
	\begin{equation}\label{inset: suggestive description}
	S_{\delta}
	\ \overset{\mathrm{def}}{=}\ 
	\big\{(u,n)\in M\!\times\!\ZZ:\langle u,v\rangle+rn\ge0\ \mbox{for all}\ (v,r)\in\delta\big\}.
	\end{equation}
Fix a monomial $a\ \!t^{n}\chi^{u}\in k[S_{\delta}]$, where $a\in k$ and $u\in M$. If $(u,0)\notin S_{\delta\cap(N_{\RR}\times\{0\})}$, then there exists some $(v',0)\in\delta\cap\big(N_{\RR}\times\{0\}\big)$ such that $\langle u,v'\rangle<0$. Thus, for some choice of $r\gg0$, we have $\langle u,v+rv'\rangle<0$. Because $\delta$ is closed under translation by $\delta\cap\big(N_{\RR}\times\{0\}\big)$, this implies that $u\in S_{\delta\cap(N_{\RR}\times\{0\})}$. This proves part (i) of the proposition.
	
	Part (ii) is a consequence of the orbit-cone correspondence for toric varieties \cite[\S3.2]{CLS}, upon observing that the special divisors in $Y_{\Delta}$ are exactly those codimension-$1$ orbit closures in $Y_{\Delta}$ that correspond to rays in $\Delta$ {\em not} lying in $N_{\RR}\times\{0\}$.
\end{proof}

\vskip .3cm

\begin{example}\label{example: toric degeneration}
	Let $M=\ZZ^{2}$, and let $\Sigma$ be the complete, $2$-dimensional fan in $N_{\RR}=\RR^{2}$ pictured at left in Figure \ref{tropcone} below. Then $Y_{\Sigma}=\mathbb{P}^{2}_{\mathbb{C}}$. Let $\Delta$ be the $3$-dimensional, $\ZZ$-admissible fan in $\RR^{2}\times\RR_{\ge0}$ pictured in Figure \ref{toric degeneration} below.
	\begin{figure}[!htb]
	$$
	\begin{xy}
	(0,0)*+{\ \ \ \ \ \includegraphics[scale=.3]{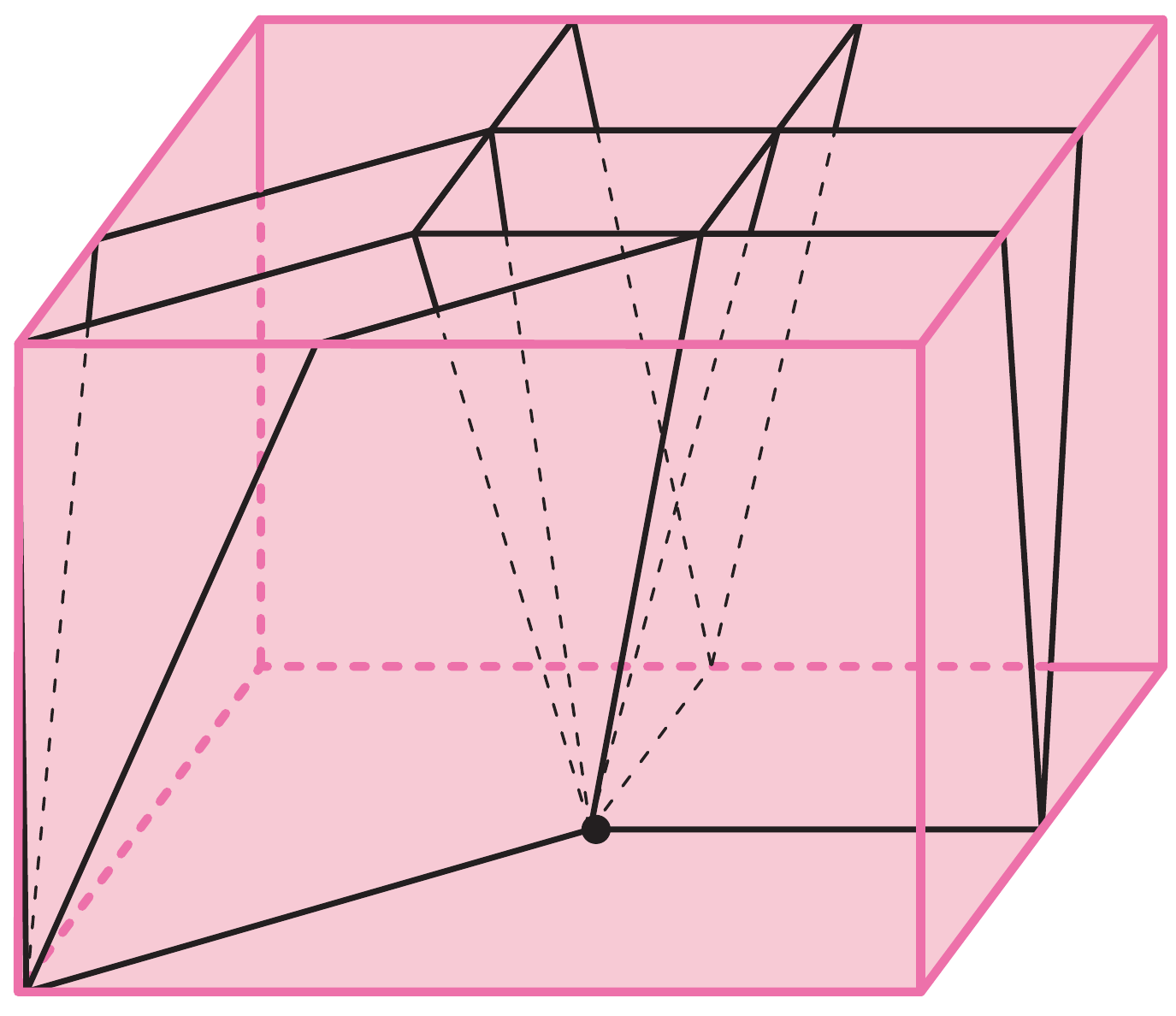}\ \ \ \ \ }="1";
	(60,0)*+{\ \ \ \ \ \includegraphics[scale=.4]{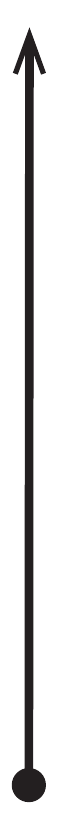}\ \ \ \ \ }="2";
	(-24,0)*+{\Delta};
	(1,-14)*+{\mbox{{\smaller $0$}}};
	(62,-16)*+{\mbox{{\smaller $0$}}};
	(69,0)*+{\big\{0,\RR_{\ge0}\big\}};
	{\ar@{->>}^{\ \ \ \ \ \ \ \ \ \ \ \ \ \ \ \ \ \ \ \ \mathrm{pr}_{2}} "1"; "2"};
	\end{xy}
	$$
	\caption{A map of fans $\mathrm{pr}_{2}:\Delta\lra\!\!\!\!\rightarrow\big\{0,\RR_{\ge0}\big\}$ inducing toric morphism of toric varieties $f_{\mathrm{pr}_{2}}:Y_{\Delta}\lra\!\!\!\!\rightarrow\mathbb{A}^{1}$.}
	\label{toric degeneration}
	\end{figure}
Because $\Delta$ satisfies $\Delta\cap\big(\RR^{2}\times\{0\}\big)=\Sigma$, it produces a complex toric variety $Y_{\Sigma}$ that comes equipped with a torus-equivariant map $Y_{\Sigma}\lra\!\!\!\!\rightarrow\mathbb{A}^{1}$. The generic fiber of the family is $\mathbb{P}^{2}$ over $k(t)$. The special fiber over $t=0$ is a union of four complex surfaces: a copy of $\mathbb{P}^{2}$, a copy of $\mathbb{P}^{1}\times\mathbb{P}^{1}$, and two Hirzebruch surfaces.
\begin{figure}[!htb]
	$$
	\underset{\mbox{The fan $\Sigma$}}{\includegraphics[scale=.325]{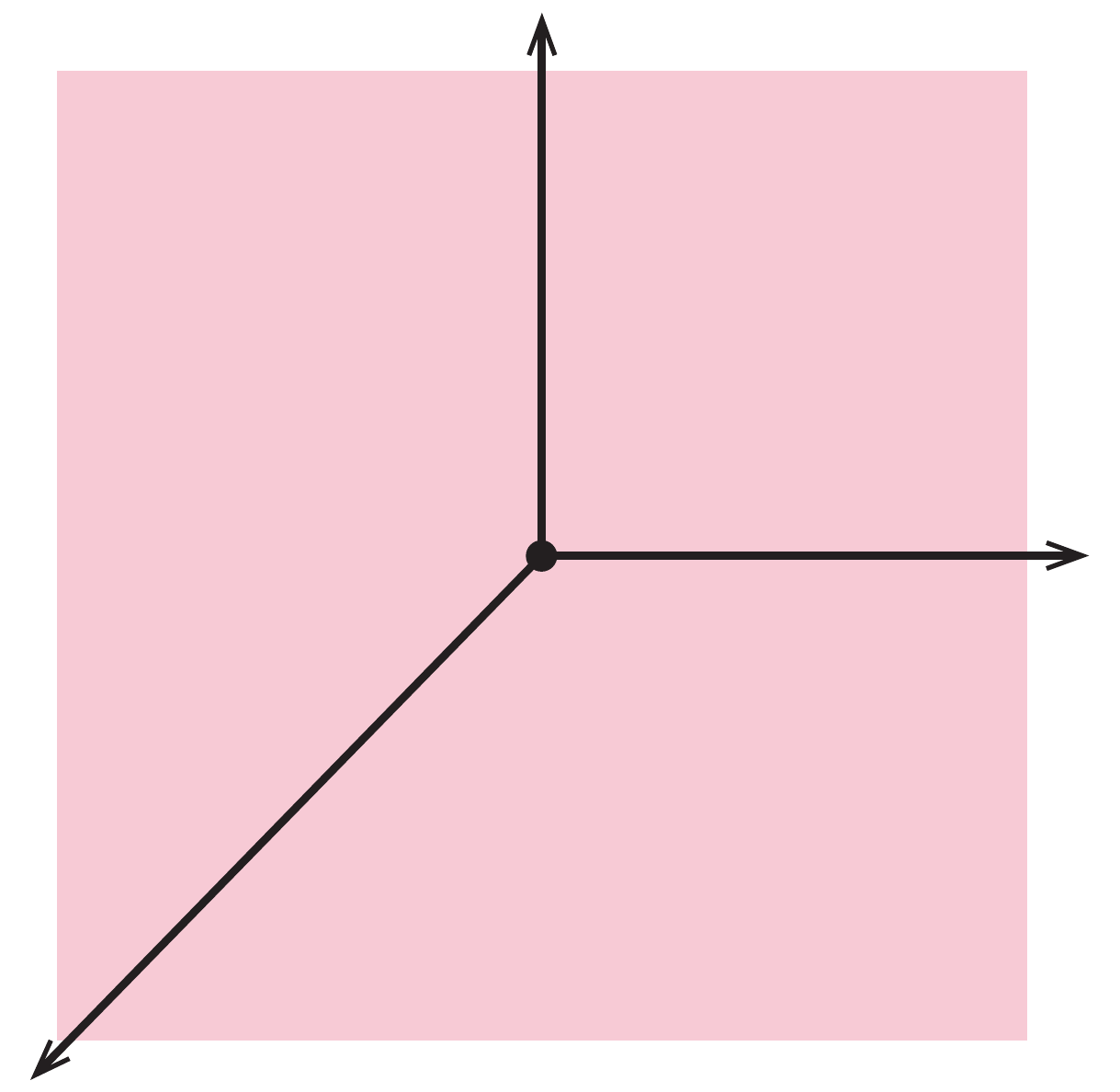}}
	\ \ \ \ \ \ \ \ \ \ \ \ \ \ \ \ \ \ \ \ \ 
	\underset{\mbox{The complex $\Delta\cap\big(\RR^{2}\times\{1\}\big)$}}{\includegraphics[scale=.325]{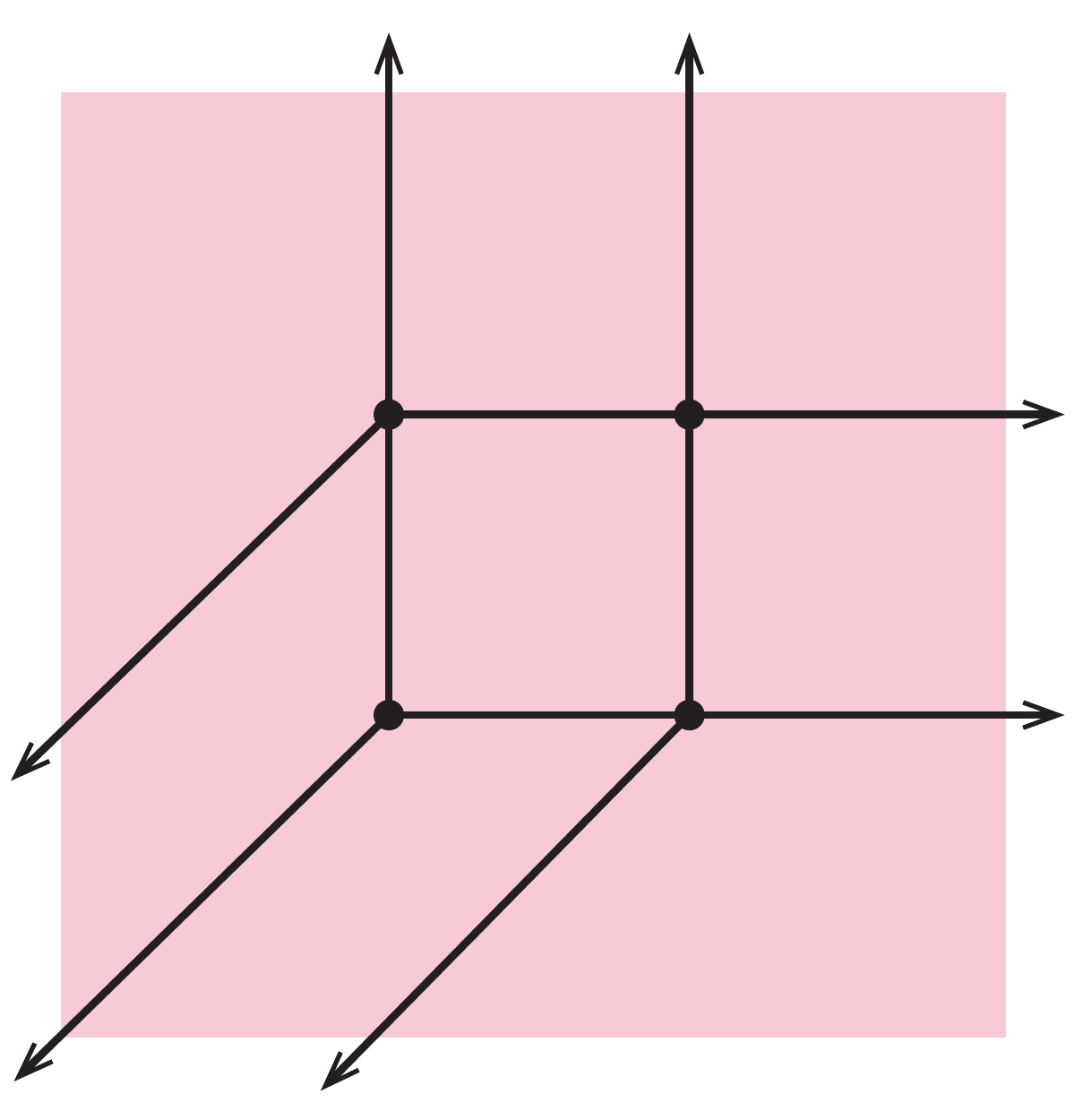}}
	$$
	\caption{The fan $\Sigma\cong\Delta\cap\big(\RR^{2}\times\{0\}\big)$ and the polyhedral complex $\Delta\cap\big(\RR^{2}\times\{1\}\big)$.}
	\label{tropcone}
	\end{figure}
These four surfaces are in one-to-one correspondence with the vertices of  the polyhedral decomposition $\Delta\cap\big(\RR^{2}\times\{1\}\big)$ pictured at left in Figure \ref{tropcone}. More precisely, they are  the toric varieties described by the stars at each vertex, and they are glued to one another along the torus-invariant divisors corresponding to the edges connecting their corresponding vertices in $C_{\Delta}$.
\end{example}
	
\vskip .3cm


\subsection{Algebraic Gubler models}\label{subsection: algebraic Gubler models}
	The fundamental observation underlying Gubler's approach to integral toric geometry, already anticipated by Kempf, Knudsen, Mumford, and Saint-Donat in \cite{KKMSD}, is that a significant part of the geometry of toric degenerations discussed in the previous \S\ref{subsection: The combinatorics of toric degenerations} works just as well upon replacing the curve $\mathbb{A}^{1}=\mathrm{Spec}_{\ \!}k[t]$ with the spectrum of the ring of integers $R$ in any non-Archimedean field $K$. This requires that we replace the factor $\ZZ$ in $M\times\ZZ$ with the value group $\Gamma=\mathrm{val}(K^{\times})\subset\mathbb{R}$ of $K$. The fact that we can express each semigroup $S_{\delta}$ appearing in the proof of Proposition \ref{proposition: how toric degenerations work} in the form (\ref{inset: suggestive description}) implies that we can express the corresponding semigroup rings as 
	\begin{equation}\label{evocative expression}
	k[S_{\delta}]\ \ =\ \ \Big\{\sum_{u\in M}a_{u}(t)\chi^{u}\in k[t,t^{-1}][M]\ :\ \langle u,v\rangle+r\ \!\mathrm{ord}_{t}\big(a_{u}(t)\big)\ge0\ \mbox{for all}\ (v,r)\in\delta\Big\},
	\end{equation}
where the coefficients $a_{u}(t)$ are Laurent polynomials in $t$, and where
	$$
	\mathrm{ord}_{t}\ \!:\ k(t)^{\times}\ \lra\ \ZZ
	$$
is the valuation on $k(t)$ given by order-of-vanishing at $t=0$. In Gubler's formalism, one replaces the ring $k[t,t^{-1}]$ in (\ref{evocative expression}) with a non-Archimedean field $K$. This produces $R$-algebras that glue to give an $R$-model of the toric $K$-variety $Y_{\Sigma}$. The present \S\ref{subsection: algebraic Gubler models} is devoted to a review of Gubler's formalism, with slight upgrades to the setting of formal $R$-schemes. In the next \S\ref{subsection: formal Gubler models}, we use the resulting formal $R$-schemes to put a topology and locally ringed structure sheaf on each exploded tropicalization.

	\vskip .3cm

\begin{definition}
A $\Gamma$-{\em admissible cone} is any convex cone
$\delta\subset N_{\mathbb{R}}\!\times\!\mathbb{R}_{\geq0}$ that can be written as an intersection of the form
$$
\delta
\ \ =\ \ 
\bigcap_{i=1}^{n}\big\{(v,c)\in N\!\times\!\mathbb{R}_{\geq0}:\langle u_{i},v\rangle+\gamma_{i}c\geq0\big\}
$$
for vectors
$(u_{1},\gamma_{1}),\dots,(u_{n},\gamma_{n})$ in $M\!\times\!\Gamma$. A $\Gamma$-{\em admissible fan} is any fan $\Delta$ in $N_{\mathbb{R}}\!\times\!\mathbb{R}_{\geq0}$ whose every cone $\delta$ is $\Gamma$-admissible.
\end{definition}
	
	\vskip .3cm
	
\begin{definition}
	If $\delta$ is a $\Gamma$-admissible cone in $N_{\mathbb{R}}\times\mathbb{R}_{\ge0}$, its associated {\em $\delta$-tilted algebra} is the ring
$$
R[U_{\delta}]
\ \ \ \overset{\mathrm{def}}{=}\ \ \ 
\Big\{\ \!\sum_{u\in M}a_{u}\chi^{u}\in K[M]\ :\langle u,v\rangle+\mathrm{val}(a_{u})\cdot c\geq0\ \mathrm{for\ all\ }(v,c)\in \delta\ \!\Big\}.
$$
\end{definition}

	\vskip .3cm
	
	By \cite[Proposition 11.3]{Gub}, every $\delta$-titled algebra $R[U_{\delta}]$ is in fact a finite-type, flat $R$-algebra. We let $U_{\delta}$ denote the $R$-scheme $U_{\delta}=\Spec{R[U_{\delta}]}$.
	
	Each inclusion $\delta'\hookrightarrow\delta$ of $\Gamma$-admissible cones in $N_{\RR}\times\RR_{\ge0}$ realizing $\delta'$ as a face of $\delta$ induces a functorial, open embedding of $R$-schemes
	\begin{equation}\label{inset: induced inclusions}
	U_{\delta'}\ \mono\ U_{\delta}.
	\end{equation}
If $\Delta$ is a $\Gamma$-admissible fan in $N_{\mathbb{R}}\times\mathbb{R}_{\ge0}$, then we can glue the affine $R$-schemes $\{U_{\delta}\}_{\delta\in\Delta}$ along these induced open embeddings (\ref{inset: induced inclusions}) to obtain a single $R$-scheme $Y_{\Delta}$. The $R$-scheme $Y_{\Delta}$ is locally of finite type over $R$. Gubler shows \cite{Gub} that when the support of $\Delta$ is $N_{\RR}\times\RR_{\ge0}$, the scheme $Y_{\Delta}$ is in fact proper over $\mathrm{Spec}_{\ \!}R$, with generic fiber canonically isomorphic to the toric $K$-variety $Y_{\Sigma}$ glued from the fan
	$$
	\Sigma\ \overset{\mathrm{def}}{=}\ \Delta\ \cap\ \big(\ \!N_{\mathbb{R}}\times\{0\}\ \!\big),
	$$
upon interpreting $\Sigma$ as a fan in $N_{\RR}$. In short:
	\begin{equation}\label{inset: non-formal generic fiber}
	(Y_{\Delta})_{K}\ \ \cong\ \ Y_{\Sigma}.
	\end{equation}
Identifying $N_{\RR}\times\{0\}$ with $N_{\RR}$, we refer to the intersection $\Delta\cap\big(N_{\mathbb{R}}\times\{0\}\big)$ as the {\em recession fan of $\Delta$ in $N_{\RR}$}, and denote it $\mathrm{rec}(\Delta)$. Just as in Proposition \ref{proposition: how toric degenerations work}.(ii), the special fiber of $Y_{\Delta}$ over the unique maximal ideal in $\mathrm{Spec}_{\ \!}R$ is a union of toric $k$-varieties indexed by the vertices in the {\em height-$1$ polyhedral complex}
	\begin{equation}\label{inset: height-1, non-Arch}
	\Delta\ \cap\ \big(\ \!N_{\mathbb{R}}\times\{1\}\ \!\big),
	\end{equation}
and glued along torus-invariant divisors according to the incidence profile of higher-dimensional polyhedra in this height-$1$ complex.
	
	We refer to $Y_{\Delta}$ as the ({\em algebraic}) {\em Gubler model} associated to the $\Gamma$-admissible fan $\Delta$.
	
	\vskip .3cm
	
\begin{example}
	Let $K=k((t))$, so that $R=k[[t]]$ and $\Gamma=\ZZ$. Then the fan $\Delta$ from Example \ref{example: toric degeneration} above, being integral, is $\Gamma$-admissible. If $\mathbb{Y}_{\Delta}$ denotes the toric $k$-variety that we constructed in Example \ref{example: toric degeneration}, equipped with its torus-equivariant fibration $f_{\mathrm{pr}_{2}}:\mathbb{Y}_{\Delta}\lra\!\!\!\!\rightarrow\mathbb{A}^{1}$, then the algebraic Gubler model associated to $\Delta$ is the fiber
	$$
	Y_{\Delta}\ \ \cong\ \ \mathrm{Spec}_{\ \!}k[[t]]\underset{\ \mathbb{A}^{1}\!}{\times}\mathbb{Y}_{\Delta}
	$$
over the $k[[t]]$-valued point $\mathrm{Spec}_{\ \!}k[[t]]\lra\mathbb{A}^{1}$ supported at $t=0$.
\end{example}
	
	\vskip .3cm
	
\begin{example}
	Let $K=\mathbb{Q}_{p}$, so that $R=\ZZ_{p}$ and $\Gamma=\ZZ$. Define $M=\ZZ$ and $N=\ZZ$, so that the standard pairing is just multiplication. Let $\Delta$ be the $\ZZ$-admissible fan in $N_{\RR}\times\RR_{\ge0}=\RR\times\RR_{\ge0}$ pictured at left in Figure \ref{figure over Z_p} below.
	\vskip -.25cm
	\begin{figure}[!htb]
	$$
	\begin{array}{ccc}
	\begin{xy}
	(0,0)*+{\ \ \ \ \ \includegraphics[scale=.4]{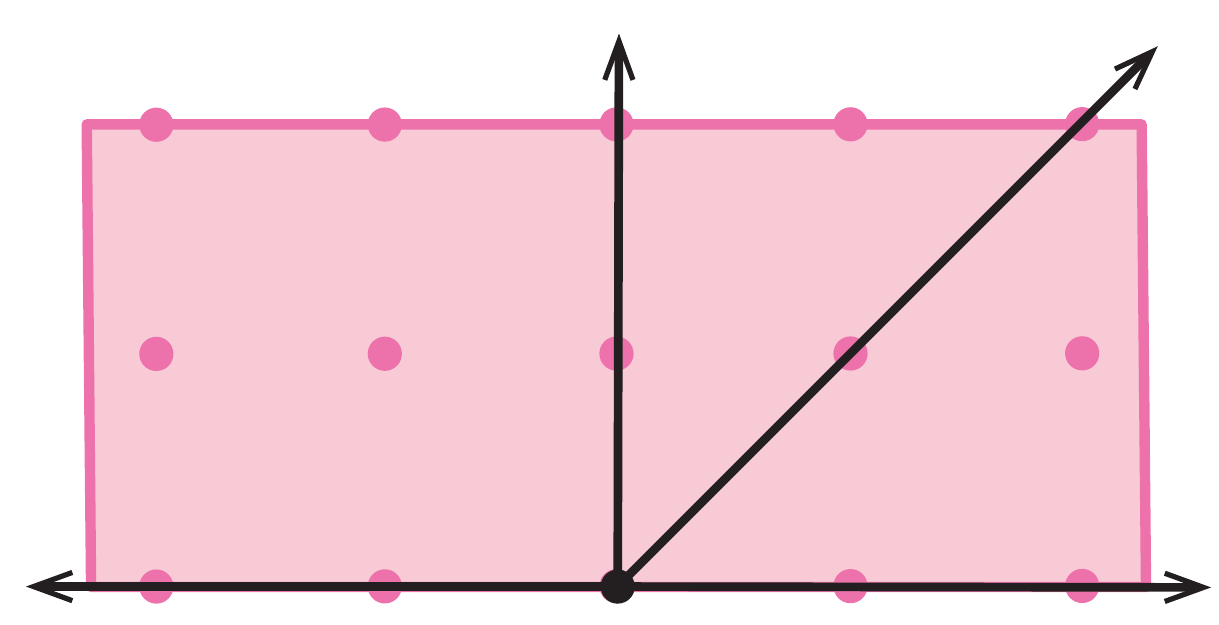}\ \ \ \ \ };
	(-25,0)*+{\Delta};
	(0,-13.5)*+{\mbox{{\smaller $0$}}};
	\end{xy}
	&
	\ \ \ \ \ \ 
	\mbox{{\larger\larger\larger $\longmapsto$}}
	\ \ 
	&
	\begin{xy}
	(-1,0)*+{\ \ \ \ \ \includegraphics[scale=.25]{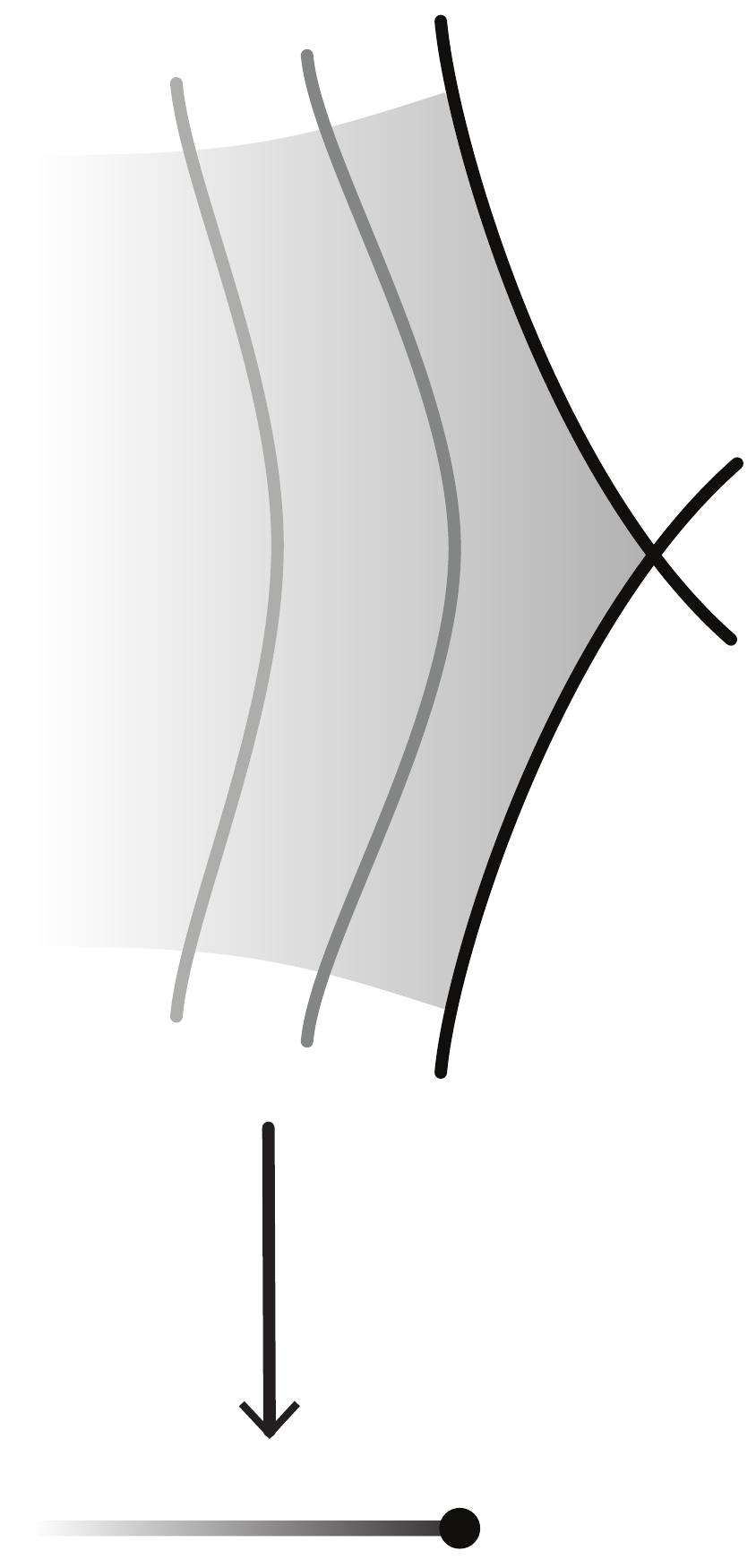}\ \ \ \ \ };
	(-9,-10.75)+*{\mathrm{Spec}_{\ \!}\ZZ_{p}};
	(3,-10.5)+*{(p)};
	(-5,3)+*{Y_{\Delta}};
	(4,-3)+*{(Y_{\Delta})_{{}_{\mathbb{F}_{p}}}};
	\end{xy}
	\end{array}
	$$
	\caption{A $\ZZ$-admissible fan in $\RR\times\RR_{\ge0}$, and its associated $\ZZ_{p}$-model of $\mathbb{P}^{1}_{\mathbb{Q}_{p}}$.}
	\label{figure over Z_p}
	\end{figure}
	
	Going from left to right, the $2$-dimensional cones in $\Delta$ have corresponding tilted algebras
	$$
	\ZZ_{p}[t^{-1}],
	\ \ \ \ \ \ \
	\ZZ_{p}[t,pt^{-1}]\ \cong\ \ZZ_{p}[x,y]\big/(xy-p),
	\ \ \ \ \ \ \ \mbox{and}\ \ \ \ \ \ 
	\ZZ_{p}[p^{-1}t].
	$$
Thus $Y_{\Delta}$ describes the projective line over $\mathbb{Q}_{p}$ degenerating to a pair of projective lines over $\mathbb{F}_{p}$ intersecting at a node.
\end{example}
	
	\vskip .3cm
	
\subsection{Formal Gubler models}\label{subsection: formal Gubler models}\ 
	If $\Delta$ is a $\Gamma$-admissible fan in $N_{\mathbb{R}}\times\RR_{\ge0}$, whose support is all of $N_{\mathbb{R}}\times\RR_{\ge0}$, then we let $\mathfrak{Y}_{\Delta}$ denote the formal $R$-scheme obtained as the completion of $Y_{\Delta}$ along its special fiber,
	$$
	\mathfrak{Y}_{\Delta}
	\ \ \overset{\mathrm{def}}{=}\ \ 
	\widehat{\ Y_{\Delta}}.
	$$
If, at each cone $\delta$ in $\Delta$, we let $R[\mathfrak{U}_{\delta}]$ denote the formal completion
	$$
	R[\mathfrak{U}_{\delta}]
	\ \ \overset{\mathrm{def}}{=}\ \ 
	\widehat{R[U_{\delta}]\ \!}
	$$
along the ideal generated by $\mathfrak{m}$ in $R[U_{\delta}]$, then $\mathfrak{Y}_{\Delta}$ is glued from the formal spectra
	$$
	\mathfrak{U}_{\delta}\ =\ \mathrm{Spf}_{\ \!}R[\mathfrak{U}_{\delta}].
	$$
Because each $\mathfrak{U}_{\delta}$ is an admissible formal $R$-scheme, this implies that $\mathfrak{Y}_{\Delta}$ is itself an admissible formal $R$-scheme.
	
	Because $\mathfrak{Y}_{\Delta}$ is a formal completion, the ``generic fiber" of $\mathfrak{Y}_{\Delta}$ is no longer a $K$-scheme as in the previous \S\ref{subsection: algebraic Gubler models}. Instead, \cite[\S2.4]{Bosch} tells us that in the formal setting, the topological $K$-algebras
	$$
	K\underset{R}{\tensor} R[\mathfrak{U}_{\delta}],
	\ \ \ \mbox{for}\ \ \ 
	\delta\ \ \mbox{in}\ \ \Delta,
	$$
are $K$-affinoid algebras. Their adic spectra
	\begin{equation}\label{inset: Berkovich spectra}
	\mathfrak{U}^{\mathrm{ad}}_{\delta}
	\ \ \overset{\mathrm{def}}{=}\ \ 
	\mathrm{Spa}\big(\ \!K\!\otimes_{R}\!R[\mathfrak{U}_{\delta}],\ R[\mathfrak{U}_{\delta}]\ \!\big)
	\end{equation}
glue to produce an {\em adic generic fiber} $\mathfrak{Y}^{\mathrm{ad}}_{\Delta}$ of the formal $R$-scheme $\mathfrak{Y}_{\Delta}$. In place of the isomorphism of $K$-varieties (\ref{inset: non-formal generic fiber}), we have a canonical isomorphism of adic spaces
	$$
	\mathfrak{Y}^{\mathrm{ad}}_{\Delta}\ \ \cong\ \ Y^{\mathrm{ad}}_{\Sigma}.
	$$
In the language of Definition \ref{admissible model of a scheme}, the formal $R$-scheme $\mathfrak{Y}_{\Delta}$ is an admissible formal model of the $K$-scheme $Y_{\Sigma}$.

\vskip .3cm


\subsection{Adic tropicalization of a closed embedding}\label{Adic tropicalization of a closed embedding}\ 

	Fix a proper algebraic variety $X$ over $K$. Even if $X$ fails to be projective, we can still ask for closed embeddings into proper toric $K$-varieties. Let us suppose that $X$ is a proper $K$-variety admiting at least one closed embedding into a proper toric $K$-variety
	$$
	\textit{\i}\ \!:\ X\ \mono\ Y_{\Sigma},
	$$
where $\Sigma$ is a complete fan in $N_{\RR}$. Using the machinery of \S\ref{subsection: algebraic Gubler models}, we can produce a flat, proper, algebraic $R$-model $Y_{\Delta}$ of $Y_{\Sigma}$ by making a choice of $\Gamma$-admissible fan $\Delta$ in $N_{\RR}\times\RR_{\ge0}$ with support equal to $N_{\RR}\times\RR_{\ge0}$. Our variety $X$ sits in the generic fiber of this model, and Gubler shows \cite{Gub} that the closure $\overline{X}$ of $X$ in $Y_{\Delta}$ is itself a flat, proper $R$-model of $X$. Let
	$$
	\mathfrak{X}_{(\Delta,\textit{\i})}
	\ \ \overset{\mathrm{def}}{=}\ \ 
	\widehat{\ \overline{X}\ }
	$$
denote the formal completion of the algebraic $R$-model $\overline{X}$ along its special fiber. This formal completion admits a canonical isomorphism
	$
	\mathfrak{X}^{\mathrm{ad}}_{(\Delta,\textit{\i})}\ \ \cong\ \ X^{\mathrm{ad}}.
	$

	\vskip .3cm
	
\begin{definition}	
	The {\em category of fans recessed over $\textit{\i}$}, denoted $\bold{REC}_{\Sigma,\textit{\i}}$, is the category with:
	\begin{changemargin}{2cm}{0cm} 
	\begin{itemize}
	\item[{\bf\em objects:}] \vskip .15cm
	Any locally finite $\Gamma$-admissible fan $\Delta$ in $N_{\mathbb{R}}\times\mathbb{R}_{\ge0}$ with recession fan $\mathrm{rec}(\Delta)=\Sigma$;
	\item[{\bf\em morphisms:}] \vskip .15cm
	Any pair of locally finite $\Gamma$-admissible fans $\Delta'$ and $\Delta$ in $N_{\mathbb{R}}\times\mathbb{R}_{\ge0}$ such that each cone $\delta'$ in $\Delta'$ is contained in at least one cone $\delta$ in $\Delta$.
	\end{itemize}
	\end{changemargin}
	\vskip .15cm
We refer to a morphism $\rho:\Delta'\lra\Delta$ in $\bold{REC}_{\Sigma,\textit{\i}}$ as a {\em refinement of $\Delta$}.
\end{definition}

\vskip .3cm
	
	If $\rho:\Delta'\lra\Delta$ is a refinement of $\Delta$, then it induces a morphism $f_{\rho}:Y_{\Delta'}\lra Y_{\Delta}$ of $R$-models of $Y_{\Sigma}$. This map of $R$-models itself restricts to a morphism
	$$
	f_{\rho}\ \!:\ \mathfrak{X}_{(\Delta',\textit{\i})}\ \lra\ \mathfrak{X}_{(\Delta,\textit{\i})}
	$$
of admissible formal $R$-models of $X$. In this way, the assignment $\Delta\mapsto\mathfrak{X}_{(\Delta,\textit{\i})}$ becomes a functor $\bold{REC}_{\Sigma,\textit{\i}}\lra\bold{AFS}_{R}^{X}$ that takes values in the category of admissible formal models of $\an{X}$.

	\vskip .3cm

\begin{definition}\label{adictrop}
{\bf (Adic tropicalization).}
The {\em adic tropicalization} of the closed embedding
$
\xymatrix{
\textit{\i}:X
\ar@{^{(}->}[r]
&
Y_{\Sigma}
}
$
is the inverse limit
$$
\big(\ \!\mathrm{Ad}(X,\textit{\i})\ \!,\ \mathscr{O}_{\!\mathrm{Ad}(X,\textit{\i})}\ \!\big)
\ \ \ \overset{\mathrm{def}}{=}\ 
\varprojlim_{\ \bold{REC}_{\Sigma,\textit{\i}}}\big(\mathfrak{X}_{_{(\Delta,\textit{\i})}},\mathscr{O}_{\mathfrak{X}_{_{(\Delta,\textit{\i})}}}\big)
$$
in the category of locally topologically ringed spaces.
\end{definition}

\vskip .3cm

	The utility of Definition \ref{adictrop} comes from the fact that, as an immediate consequence of Remark \ref{remark: Gillam}, it gives $\mathrm{Ad}(X,\textit{\i})$ the structure of a locally topologically ringed topological space. But Definition \ref{adictrop} obscures the relationship between $\mathrm{Ad}(X,\textit{\i})$ and the exploded tropicalization of $X$. To clarify this relationship, first consider the special case where $\textit{\i}:X'\mono\mathbb{T}'$ is a closed embedding into a torus, with $N'$ the cocharacter lattice of $\mathbb{T}'$. In this case, the {\em exploded tropicalization} is as described in \S\ref{introduction}: it is the union
	$$
	\mathfrak{Trop}(X',\textit{\i}')
	\ \ \overset{\mathrm{def}}{=}\ \ 
	\bigsqcup_{v\in N'_{\RR}}|\mathrm{in}_{v}X'|.
	$$
For a closed embedding $\textit{\i}:X\mono Y_{\Sigma}$ into an arbitrary toric variety $Y_{\Sigma}$, the {\em extended exploded tropicalization}, denoted $\mathfrak{Trop}(X,\textit{\i})$, is the set-theoretical union
	$$
	\mathfrak{Trop}(X,\textit{\i})\ \ =\ \ \bigsqcup_{\sigma\in\Sigma}\mathfrak{Trop}\big(\ \!\textit{\i}^{-1}\big(O(\sigma)\big),\ \!\textit{\i}\ \!\big),
	$$
where $\textit{\i}^{-1}\big(O(\sigma)\big)$ is the part of $X$ that maps to the torus-orbit $O(\sigma)\cong\mathbb{T}_{N/\sigma\cap N}$ under $\textit{\i}$ (see \cite[\S3.2]{CLS}). In other words, the extended exploded tropicalization is the union of all exploded tropicalizations encoded in $\textit{\i}$ as we run over all locally closed torus-orbit strata in $Y_{\Sigma}$.
	
	Note that each $\Gamma$-rational vector $v$ in $N_{\RR}$ determines a $\Gamma$-admissible ray $\tau_{v}\overset{{}_{\mathrm{def}}}{=}\RR_{\ge0}\ \!(v,1)$ in $N_{\RR}\times\RR_{\ge0}$, and vice versa, as depicted at left  in Figure \ref{points and rays}. When $v$ is {\em not} $\Gamma$-rational, the ray $\tau_{v}$ is no longer $\Gamma$-admissible, but there are still many $\Gamma$-admissible cones $\delta$ containing $\tau_{v}$, as depicted at right in Figure \ref{points and rays}.
	\begin{figure}[!htb]
	$$
	\ \ \ \ \ \ \ \ \ 
	\scalebox{.85}{$
	\begin{xy}
	(0,0)*+{\includegraphics[scale=.35,clip=true,trim=85pt 172pt 150pt 200pt]{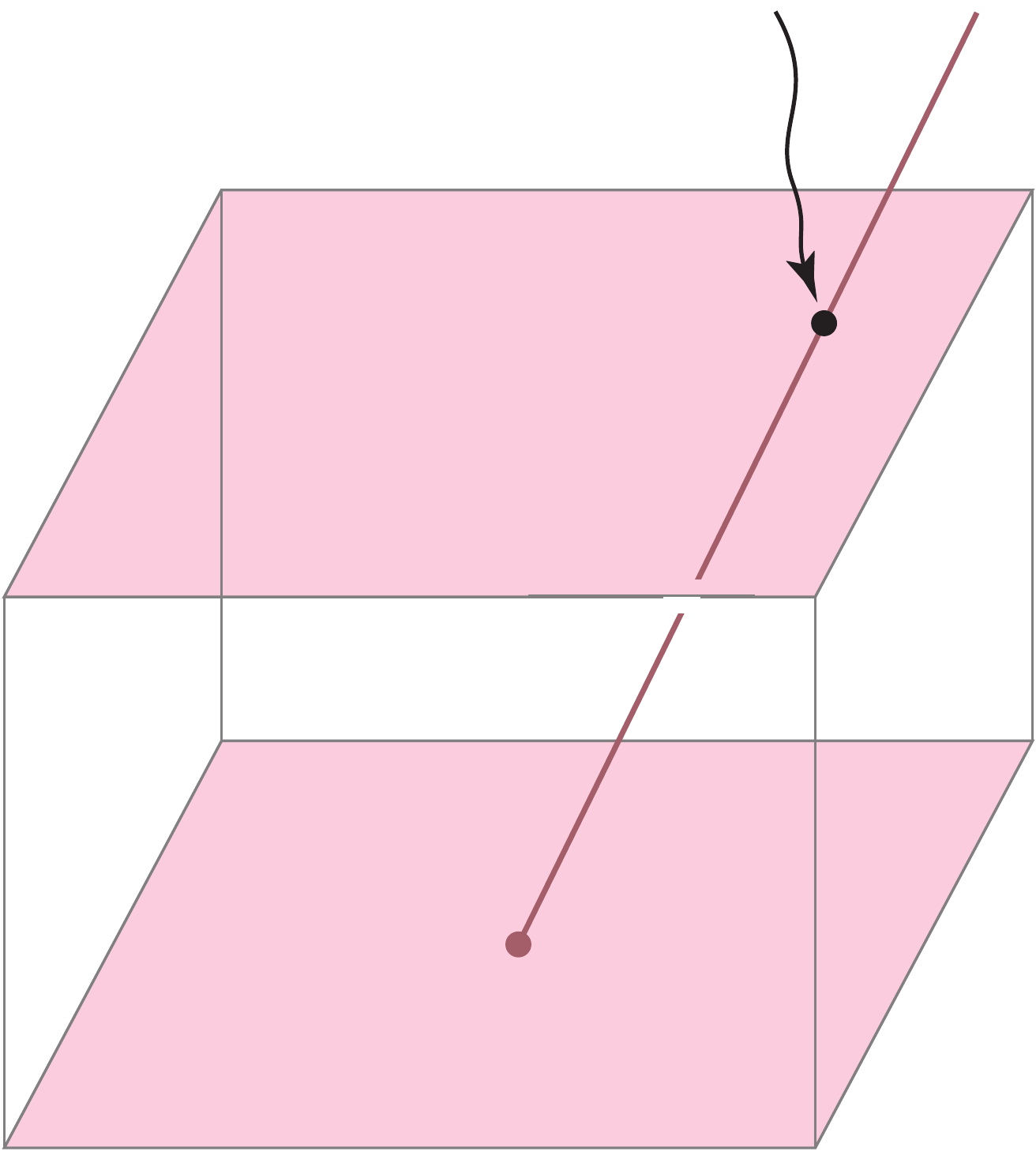}};
	(10,28)*+{(v,1)};
	(28,29)*+{\ \ \ \ \ \ \ \ \tau_{v}\ \!\overset{\mathrm{def}}{=}\ \!\mathbb{R}_{\ge0}(v,1)};
	\end{xy}
	\ 
	\begin{xy}
	(0,0)*+{\includegraphics[scale=.35,clip=true,trim=80pt 171pt 101pt 210pt]{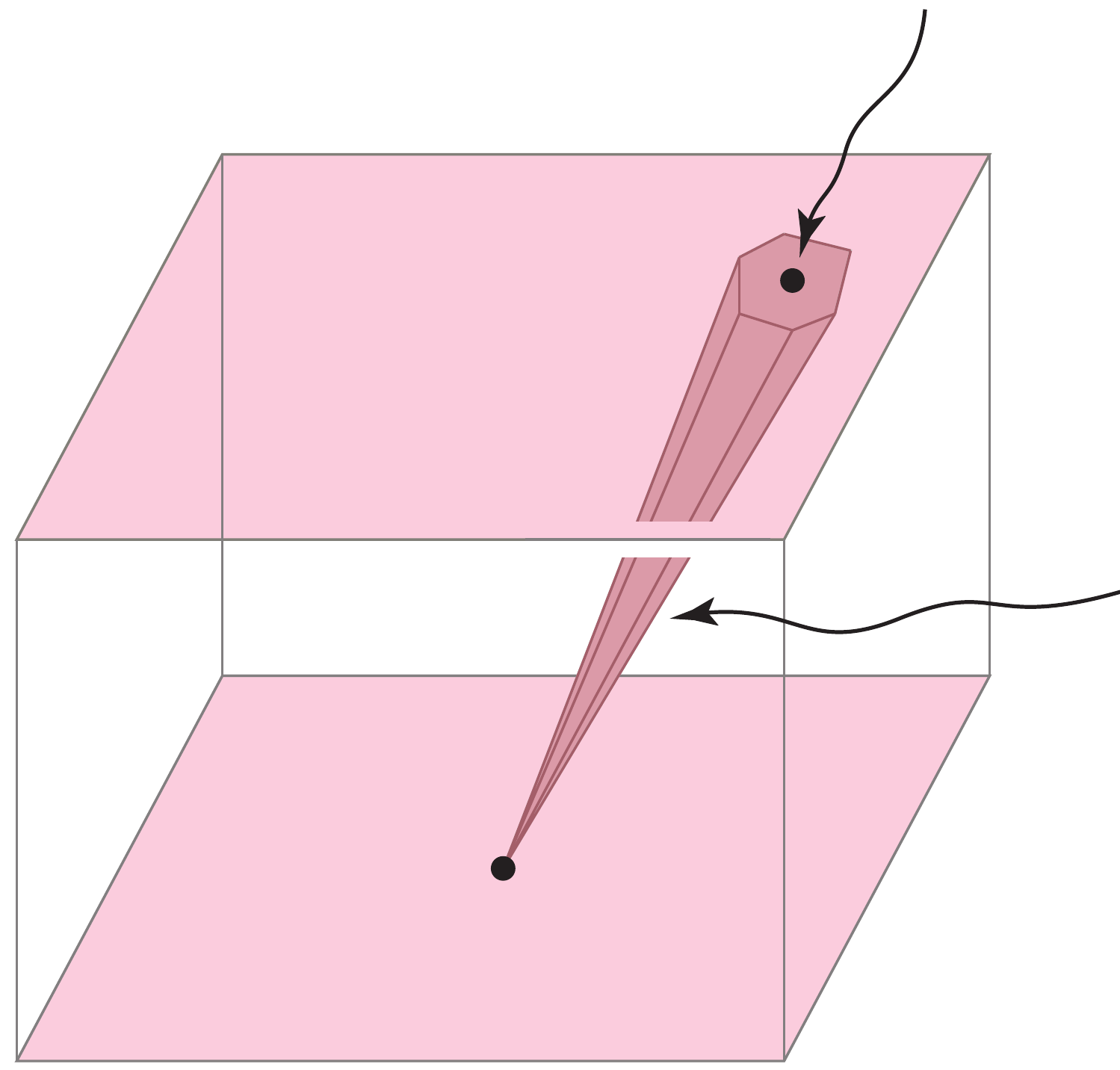}};
	(17,28)*+{(v,1)};
	(42.5,2)*+{\mbox{$\Gamma$-admissible cone}};
	(42.5,-2.75)*+{\delta\mbox{ containing the}};
	(42.5,-7)*+{\mbox{ray $\tau_{v}=\mathbb{R}_{\ge0}(v,1)$}};
	\end{xy}
	$}
	$$
	\caption{The ray $\tau_{v}$ spanned by a vector $(v,1)$ in $N_{\RR}\times\{1\}$, at left, and a $\Gamma$-admissible cone $\delta$ containing $\tau_{v}$ when $\tau_{v}$ is not itself $\Gamma$-rational.}
	\label{points and rays}
	\end{figure}

\vskip .3cm

\begin{proposition}\label{exploded lemma 1}
For any vector $v$ in $N_{\mathbb{R}}$, we have a natural isomorphism
	\begin{equation}\label{inset: inverse limit giving initial degeneration}
	\mathrm{in}_{v}X
	\ \ \cong\ \ 
	\varprojlim_{\delta\supset\tau_{v}}\big(\mathfrak{X}_{\delta}\big)_{k},
	\end{equation}
where the limit runs over all $\Gamma$-admissible cones $\delta$ in $N_{\mathbb{R}}\times\mathbb{R}_{\ge0}$ that contain the ray $\tau_{v}$. When $v$ is $\Gamma$-rational, the above isomorphism (\ref{inset: inverse limit giving initial degeneration}) simplifies to an isomorphism
	$$
	\mathrm{in}_{v}X\ \ \cong\ \ \big(\mathfrak{X}_{\tau_{v}}\big)_{k}.
	$$
\end{proposition}

\vskip .3cm

	One uses Proposition \ref{exploded lemma 1} to prove that the topological space underlying $\mathrm{Ad}(X,\textit{\i})$ coincides with the extended exploded tropicalization $\mathfrak{Trop}(X,\textit{\i})$:

\vskip .3cm

\begin{theorem}\label{theorem: topology on exploded trop}
	Given any closed embedding $\textit{\i}:X\mono Y_{\Sigma}$ of a $K$-variety $X$ into a proper toric $K$-variety $Y_{\Sigma}$, there is a natural bijection 
	$$
	\mathrm{Ad}(X,\textit{\i})\ \xrightarrow{\ \sim\ }\ \mathfrak{Trop}(X,\textit{\i})
	$$
from the underlying set of the adic tropicalization to the exploded extended tropicalization.
\end{theorem}

\vskip .3cm


\subsection{Adic tropicalization and metrized complexes}\ 
	
	In this section, we describe the adic tropicalization of a generic hyperplane in the toric variety $\mathbb{P}^{2}$ in detail, and we make some remarks about the relationship between adic tropicalizations and the metrized complexes of O. Amini and M. Baker \cite{Amini-Baker}.

\vskip .3cm

\subsubsection{{\bf Tropical complexes and metrized complexes}}

Fix a $K$-variety $X$ and a closed embedding $\textit{\i}:X\betterhookarrow Y_{\Sigma}$ into a toric variety. Each complete fan $\Delta$ in $N_{\mathbb{R}}\!\times\!\mathbb{R}_{\ge0}$, with recession fan $\mathrm{rec}(\Delta)=\Sigma$, determines a formal $R$-scheme $\mathfrak{X}_{\Delta}$. As explained in \S\ref{specializations of adic spaces}, the associated adic space $\mathfrak{X}^{\mathrm{ad}}_{\Delta}\cong X^{\mathrm{ad}}$ comes with a specialization morphism
	$$
	\mathrm{sp}_{\mathfrak{X}_{\Delta}}\ \!:\ 
	X^{\mathrm{ad}}
	\ \lra\ 
	\mathfrak{X}_{\Delta}.
	$$
Because the Berkovich analytification $X^{\mathrm{an}}$ is the maximal Hausdorff quotient $h:X^{\mathrm{ad}}\lra\!\!\!\!\rightarrow X^{\mathrm{an}}$, we also have a map
	$$
	q\ \!:\ X^{\mathrm{ad}}\ \xrightarrow{\ \ h}\!\!\!\!\rightarrow\ X^{\mathrm{an}}\ \xrightarrow{\mathrm{trop}}\ \mathrm{Trop}(X,\textit{\i}).
	$$
Define the {\em tropical complex} ({\em of varieties}) {\em associated to $\Delta$}, denoted $\mathfrak{CX}_{\Delta}$, to be the image, inside the product of topological spaces $\mathrm{Trop}(X,\textit{\i})\times\mathfrak{X}_{\Delta}$, of the product map
	$$
	q\times\mathrm{sp}_{\mathfrak{X}_{\Delta}}
	\ \!:\ 
	X^{\mathrm{ad}}
	\ \lra\ 
	\mathrm{Trop}(X,\textit{\i})\times\mathfrak{X}_{\Delta}.
	$$
	
\vskip .3cm

\begin{example}\label{metrized complex example}
{\bf Metrized complex associated to a generic hyperplane in $\pmb{\mathbb{P}^{2}}$.}
For a concrete example, let $Y_{\Sigma}=\mathbb{P}^{2}$ with its standard toric structure, let $\textit{\i}:X\!\xymatrix{{}\ar@{^{(}->}[r]&{}}Y\!$ be the closed embedding of the hyperplane $X=V(x+y+1)$ into $\mathbb{P}^{2}$, and let $C$ be any polyhedral complex decomposing $N_{\mathbb{R}}=\mathbb{R}^{2}$ such that the intersection $C_{X}\overset{{}_{\mathrm{def}}}{=}C\cap\mathrm{Trop}(X,\textit{\i})$ is the one pictured at left in Figure \ref{metrized complex} below. Then our tropical complex $\mathfrak{CX}_{\Delta_{C}}$ is the {\em metrized complex} ({\em of curves}), in the sense of Amini and Baker \cite[\S1.2]{Amini-Baker}, pictured at right in Figure \ref{metrized complex}.

\begin{figure}[!htb]
$$
\vcenter{\hbox{
\begin{xy}
(0,0)*+{\includegraphics[scale=.75]{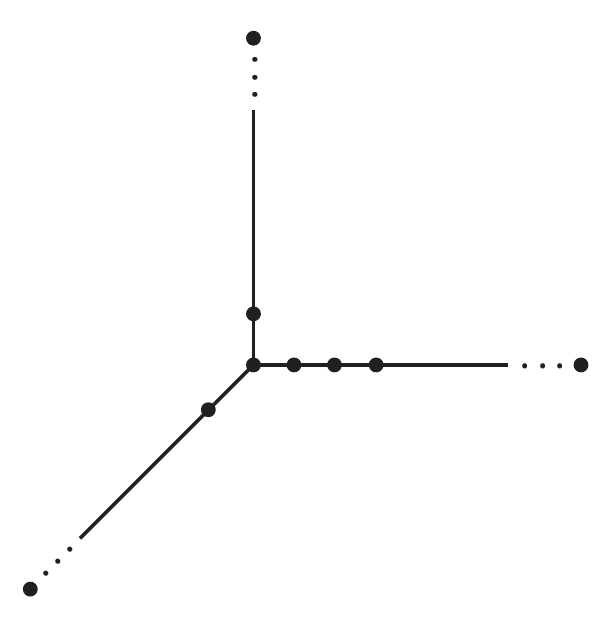}};
(4,-15)*+{C_{X}\overset{{}_{\mathrm{def}}}{=}C\cap\mathrm{Trop}(X,\textit{\i})\!\!\!\!\!\!\!\!\!\!\!\!\!};
(0,28)*+{{\color{white}.\color{black}}};
\end{xy}
}}
\ \ \ \ \ \ \ \ \ \ \ \ \ \ \ \raisebox{-14pt}{\huge $\xmapsto{\ \ \ }$}\ \ 
\vcenter{\hbox{
\begin{xy}
(0,-5)*+{\includegraphics[scale=.75]{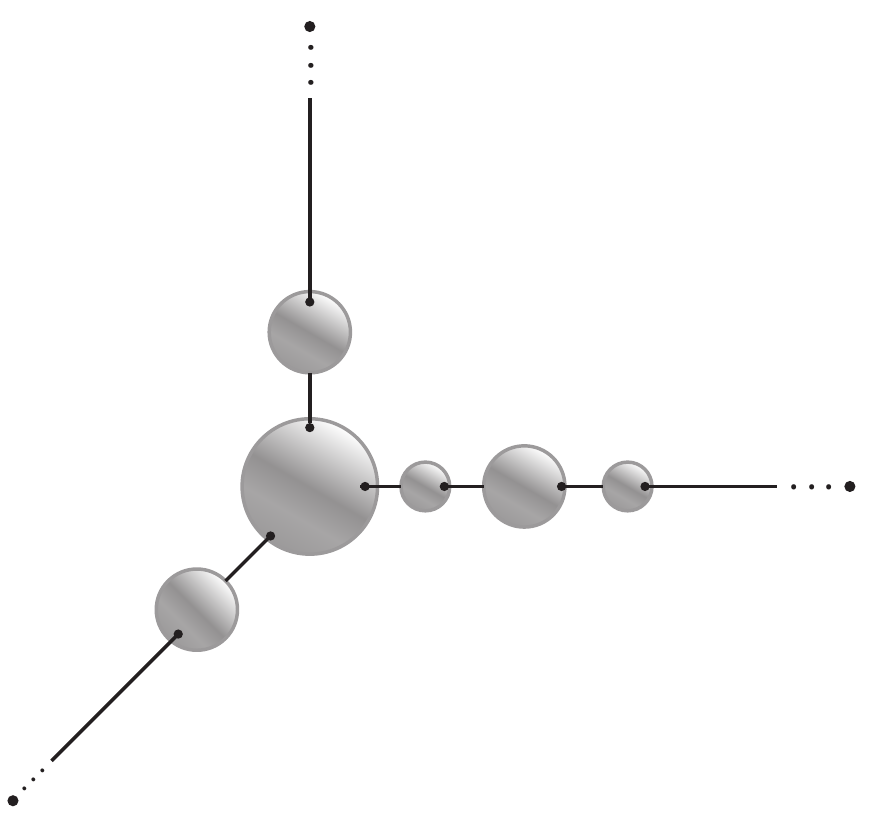}};
(-3,-20)*+{\mathfrak{CX}_{\Delta_{C}}};
\end{xy}
}}
$$
\caption{A $\Gamma$-rational polyhedral decomposition of $\mathrm{Trop}(X)$ induced by a convex polyhedral decomposition $C$ of $N_{\mathbb{R}}$, and its associated tropical complex $\mathfrak{CX}_{\Delta_{C}}$.}\label{metrized complex}
\end{figure}
\end{example}

\vskip .3cm
 	
\subsubsection{{\bf The general 1-dimesional case and metrized complexes.}}
Consider the case where our embedding $\textit{\i}:X\betterhookarrow Y_{\Sigma}$ is a closed embedding of a smooth curve $X$ into the toric variety $Y_{\Sigma}$. Recall that $\textit{\i}$ is {\em sch\"on} if the initial degeneration $X_{v}$ is smooth at every vector $v$ in the tropicalization of each torus orbit in $Y_{\Sigma}$.
	
	We claim if $Y_{\Sigma}$ is proper and $\textit{\i}$ is sch\"on, then each tropical complex $\mathfrak{CX}_{\Delta}$ is a metrized complex (of curves) in the sense of Amini and Baker \cite[\S1.2]{Amini-Baker}. Indeed, Lemma \ref{exploded lemma 1}, in conjunction with the sch\"on condition, implies that if $\delta$ is a cone in $\Delta$ whose intersection $\delta\cap\big(\mathrm{Trop}(X,\textit{\i})\times\{1\}\big)$ is a (necessarily $\Gamma$-rational) point $(v,1)\in\mathrm{Trop}(X,\textit{\i})\times\{1\}$, then the $k$-scheme underlying $\mathfrak{X}_{\delta}$ is the nonsingular curve $\mathrm{in}_{v}X$. Any point in $X^{\mathrm{ad}}$ that gets mapped to $v$ under $q:X^{\mathrm{ad}}\lra\mathrm{Trop}(X,\textit{\i})$ goes to a point of this nonsingular curve $\mathrm{in}_{v}X$ under the specialization map $\mathrm{sp}_{\mathfrak{X}_{\Delta}}\!:X^{\mathrm{ad}}\lra\mathfrak{X}_{\Delta}$. Similarly, Lemma \ref{exploded lemma 1} and the sch\"on condition imply that if $\delta\cap\big(\mathrm{Trop}(X,\textit{\i})\times\{1\}\big)$ is an edge connecting points $(u,1)$ and $(v,1)$ in $\mathrm{Trop}(X,\textit{\i})\times\{1\}$, then the $k$-scheme underlying $\mathfrak{X}_{\delta}$ is a $k$-curve with two smooth components meeting at a point, one of the components containing $\mathrm{in}_{u}X$ and the other containing $\mathrm{in}_{v}X$. The points of $\mathrm{X}^{\mathrm{ad}}$ mapping to $\mathfrak{X}_{\delta}-\big(\mathrm{in}_{u}X\cup\mathrm{in}_{v}X\big)$ under the specialization map $\mathrm{sp}_{\mathfrak{X}_{\Delta}}$ are mapped to interior points of the edge corresponding to $\delta$ inside $\mathrm{Trop}(X,\textit{\i})$.
	
	In this way, we can describe the tropical complex $\mathfrak{CX}_{\Delta}$ using exactly the kind of gluing datum that Amini and Baker use to define a metrized complex in \cite[\S1.2]{Amini-Baker}.

\vskip .3cm

\begin{example}\label{2-dimensional tropical complex}
{\bf An example of a 2-dimensional tropical complex.}
Let $X=Y_{\Sigma}=\mathbb{P}^{2}$, where $\Sigma\subset N_{\mathbb{R}}=\mathbb{R}^{2}$ is the standard fan describing $\mathbb{P}^{2}$ as a toric variety. Let $C$ be the polyhedral decomposition of $N_{\mathbb{R}}=\mathbb{R}^{2}$ pictured in Figure \ref{tropical complex} below. It determines a fan $\Delta_{C}$ in $N_{\mathbb{R}}\!\times\!\mathbb{R}_{\ge0}$ with recession fan $\mathrm{rec}(\Delta_{C})=\Sigma$. Using Lemma \ref{exploded lemma 1}, we can build the tropical complex $\mathfrak{CX}_{\Delta_{C}}$ from the datum of the $k$-varieties underlying the formal models $\mathfrak{X}_{\delta_{P}}$ associated to  each polygon $P$ in $C$.
\begin{figure}[!htb]
$$
\begin{xy}
(0,0)*+{\includegraphics[scale=.55,clip=true,trim=80pt 310pt 50pt 290pt]{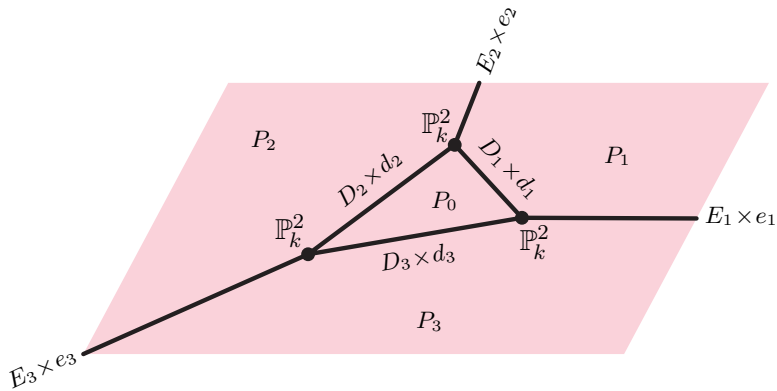}};
(1,12)*+{\mathbb{P}^{2}_{\!k}};
(14,-3.5)*+{\mathbb{P}^{2}_{\!k}};
(-18.5,-2)*+{\mathbb{P}^{2}_{\!k}};
(-8,5)*+{\rotatebox{40}{{\smaller\smaller $D_{2}\!\times\!d_{2}$}}};
(11,6)*+{\rotatebox{-45}{{\smaller\smaller $\!\!D_{1}\!\times\!d_{1}$}}};
(-1,-5.5)*+{\rotatebox{10}{{\smaller\smaller $\!\!D_{3}\!\times\!d_{3}$}}};
(9,24)*+{\rotatebox{70}{{\smaller\smaller $\!\!E_{2}\!\times\!e_{2}$}}};
(42,-0)*+{\rotatebox{0}{{\smaller\smaller $\!\!E_{1}\!\times\!e_{1}$}}};
(-51,-20)*+{\rotatebox{25}{{\smaller\smaller $\!\!E_{3}\!\times\!e_{3}$}}};
(2,2)*+{\mbox{{\smaller\smaller $P_{0}$}}};
(25,8)*+{\mbox{{\smaller\smaller $P_{1}$}}};
(-22,10)*+{\mbox{{\smaller\smaller $P_{2}$}}};
(0,-14)*+{\mbox{{\smaller\smaller $P_{3}$}}};
\end{xy}
$$
\caption{The tropical complex $\mathfrak{CX}_{\Delta_{C}}$ in Example \ref{2-dimensional tropical complex}.}
\label{tropical complex}
\end{figure}

The polyhedral complex $C$ consists of three vertices, edges $d_{i}$ and $e_{i}$ for $1\le i\le 3$, and 2-dimensional polygons $P_{j}$ for $0\le j\le 3$. The tropical complex $\mathfrak{CX}_{\Delta_{C}}$ is glued from a copy of $\mathbb{P}^{1}_{k}$ at each vertex, a product $D_{i}\!\times\!d_{i}$ with a $1$-dimensional $k$-variety $D_{i}$ associated to each edge $d_{i}$, $1\le i\le 3$, and similarly for edges $e_{i}$, and then the three $2$-dimensional polygons $P_{j}$, $1\le j\le 3$, understood as products $\mathrm{Spec}_{\ \!}k\times P_{j}$. The data telling us how to glue these topological spaces comes from the incidence relations between the irreducible components of the $k$-variety underlying the formal model $\mathfrak{X}_{\Delta_{C}}$.
\end{example}

\vskip .3cm

\begin{lemma}\label{tropical complex description of adic trop}
If $\textit{\i}:X\betterhookarrow Y_{\Sigma}$ is a closed embedding into a proper toric variety, then there is a canonical bijection $\varprojlim_{\bold{REC}_{\Sigma,\textit{\i}}}\!\!\mathfrak{CX}_{\Delta}\ \xrightarrow{\ \sim\ }\ \mathrm{Ad}(X,\textit{\i})$.
\end{lemma}
\begin{proof}
Each projection $\mathrm{Trop}(X,\textit{\i})\times\mathfrak{X}_{\Delta}\lra\!\!\!\!\rightarrow\mathfrak{X}_{\Delta}$ induces a map $\mathfrak{CX}_{\Delta}\lra\mathfrak{X}_{\Delta}$, and these maps give rise to a map
	\begin{equation}\label{comparison map tropical complex to adic trop}
	\varprojlim_{\bold{REC}_{\Sigma,\textit{\i}}}\!\!\mathfrak{CX}_{\Delta}\ \lra\ \mathrm{Ad}(X,\textit{\i}).
	\end{equation}
	
	Fix a compatible system of points $x=(x_{\Delta})_{\Delta\in\mathrm{Rec}_{\Sigma,\textit{\i}}}$, where $x_{\Delta}$ lies in the $k$-variety underlying $\mathfrak{X}_{\Delta}$. This system describes a point $x$ in $\mathrm{Ad}(X,\textit{\i})$. Because $\mathfrak{X}_{\delta_{1}}\cap\mathfrak{X}_{\delta_{2}}=\mbox{\O}$ in $\mathfrak{X}_{\Delta}$ whenever $\delta_{1}\cap\delta_{2}=\mbox{\O}$ in $N_{\mathbb{R}}\!\times\!\mathbb{R}_{\ge0}$, we know that the compatible system $x$ has an associated point $v_{x}\in\mathrm{Trop}(X,\textit{\i})$, determined by the condition that $v_{x}$ lie in the polygon $P_{\delta}$ associated to a cone $\delta$ in $\Delta$ whenever $x_{\Delta}$ lies in $\mathfrak{X}_{\delta}$. Here $P_{\delta}$ denotes the unique polygon in $N_{\mathbb{R}}$ satisfying $P_{\delta}\!\times\!\{1\}=\delta\cap\big(N_{\mathbb{R}}\!\times\!\{1\}\big)$. The very definition of $v_{x}$ implies that if $x'$ is any point of $\mathfrak{X}^{\mathrm{ad}}$ mapping to $x$ under $\mathrm{sp}_{\mathfrak{X}_{\Delta}}:X^{\mathrm{ad}}\lra\mathfrak{X}_{\Delta}$, then $x'$ maps to $v_{x}$ under the map $q:X^{\mathrm{ad}}\lra\mathrm{Trop}(X,\textit{\i})$. This implies subjectivity of the comparison map (\ref{comparison map tropical complex to adic trop}).
	
	To see that (\ref{comparison map tropical complex to adic trop}) is injective, note that the argument of the previous paragraph implies that points given by pairs $(x,u)$ and $(y,v)$ in $\mathfrak{CX}_{\Delta}$ must map to distinct points of $\mathrm{Ad}(X,\textit{\i})$ if their second coordinates $u,v\in\mathrm{Trop}(X,\textit{\i})$ differ. On the other hand, if $u=v$, but the inverse systems $y=(y_{\Delta})_{\Delta\in\mathrm{Rec}_{\Sigma,\textit{\i}}}$ and $x=(x_{\Delta})_{\Delta\in\mathrm{Rec}_{\Sigma,\textit{\i}}}$ differ, then they map to distinct points of $\mathrm{Ad}(X,\textit{\i})$.
\end{proof}

\vskip .3cm

\begin{example}
{\bf The adic tropicalization of a generic hyperplane in $\pmb{\mathbb{P}^{2}}$.}
As in Example \ref{metrized complex example} above, let $\textit{\i}:X\!\xymatrix{{}\ar@{^{(}->}[r]&{}}Y\!$ be the closed embedding of the hyperplane $X=V(x+y+1)$ into $\mathbb{P}^{2}$. Lemma \ref{tropical complex description of adic trop} says that we can figure out how $\mathrm{Ad}(X,\textit{\i})$ looks by undrestanding it as the object that results, in the inverse limit, when we bubble of a copy of $\mathbb{P}^{1}_{k}$ at every $\Gamma$-rational point in $\mathrm{Trop}(X,\textit{\i})$. The resulting topological space appears in Figure \ref{adic trop drawing} below.

\begin{figure}[!htb]
$$
\includegraphics[scale=.625]{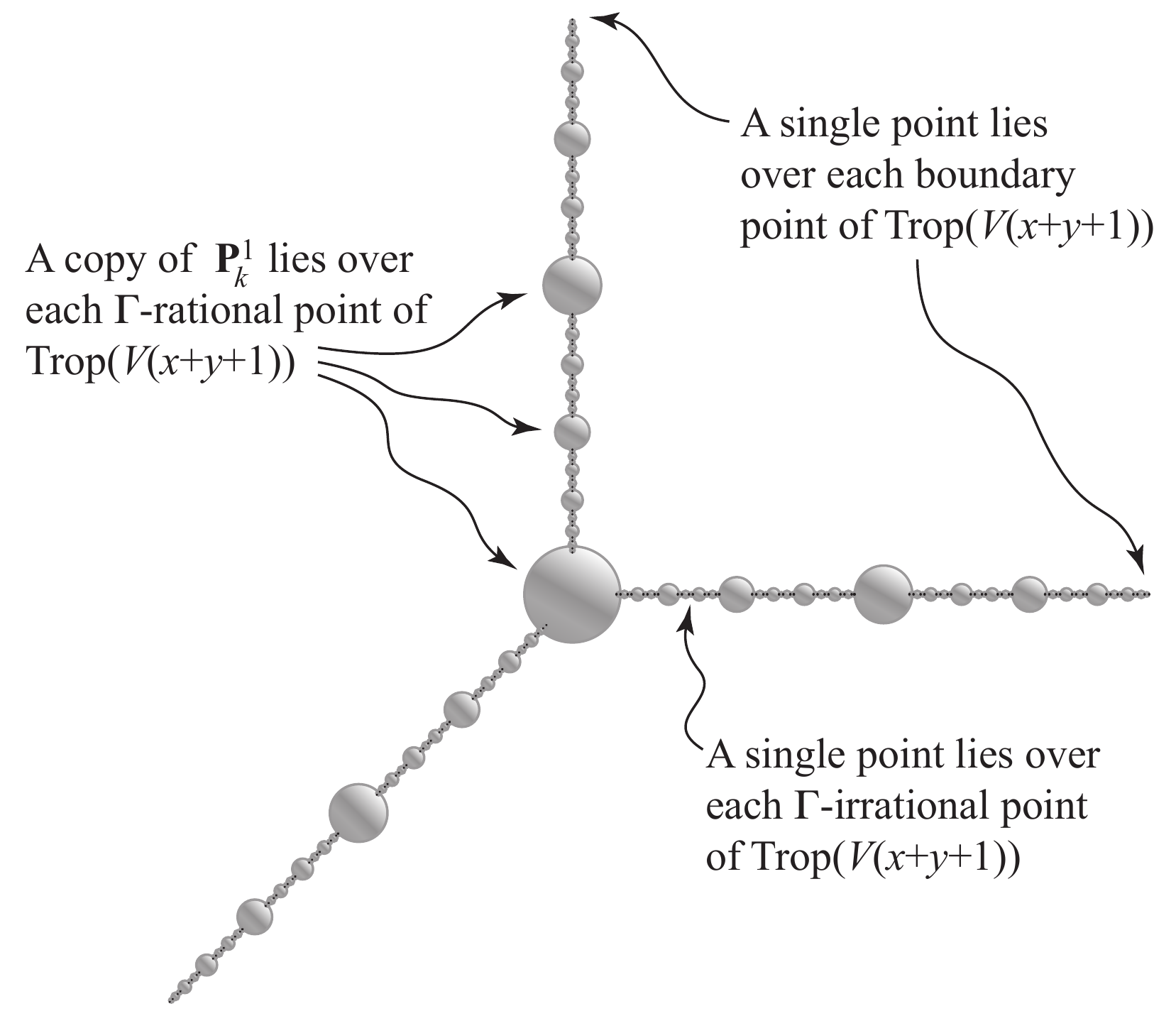}
$$
\caption{Adic tropicalization of the hyperplane $V(x+y+1)$ in $\mathbb{P}^{2}$. Compare this to the exploded curve, in the sense of B. Parker \cite{Parker2}, appearing in Figure \ref{figure: exploded curve}.}
\label{adic trop drawing}
\end{figure}
\end{example}

\vskip .3cm

\subsubsection{{\bf Tropical ``type 5" points and inverse systems of $R$-models.}}\label{tropical type five points}
We take a moment to highlight a phenomenon that occurs in adic tropicalizations, related to the presence of ``type 5" points in the adic affine line $(\mathbb{A}^{1})^{\mathrm{ad}}$.
	
	Consider a polygonal interval $P_{1}=\begin{tikzpicture}\draw[black, thick] (0,0) -- (1,0);\fill[black] (0,0) circle (.08);\fill[black] (1,0) circle (.08);\end{tikzpicture}$ inside $N_{\mathbb{R}}=\mathbb{R}$. The $k$-variety underlying the model $\mathfrak{X}_{P_{1}}$ consists of two copies of $\mathbb{A}^{1}_{k}$ intersecting at a single $k$-point $x_{1}$ as pictured in Figure \ref{bubbling off} below. If we subdivide $P_{1}$ by adding a single $\Gamma$-rational vertex in the interior of $P_{1}$, then the $k$-variety underlying our formal $R$-scheme picks up a $\mathbb{P}^{1}_{\!k}$ component where $x_{1}$ used to be.
	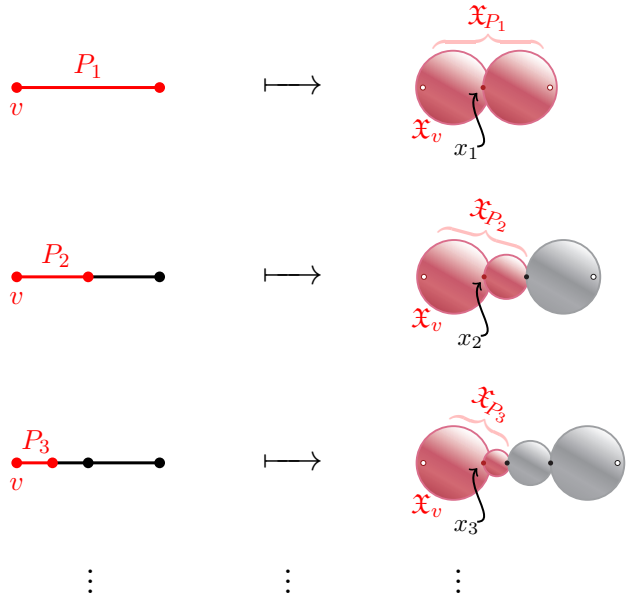
\begin{figure}[!htb]
	$$
	\scalebox{.95}{$
	\begin{array}{rcl}
	\begin{xy}
	(0,0)*+{\begin{tikzpicture}
		\draw[red, very thick] (0,0) -- (2,0);
		\fill[red] (0,0) circle (.075);
		\fill[red] (2,0) circle (.075);
	\end{tikzpicture}};
	(0,3)*+{{\color{red} P_{1}}};
	(-10,-3)*+{{\color{red} v}};
	\end{xy}
	&
	\ \ \ \ \ \ \ \mbox{\larger\larger $\xmapsto{\ \ \ }$}\ \ \ \ \ \ 
	&
	\begin{xy}
	(0,0)*+{\includegraphics[scale=.75]{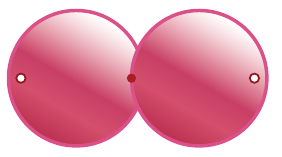}};
	(0,8.5)*+{{\color{red} \overset{\mbox{$\mathfrak{X}_{\!P_{1}}$}}{{\color{pink}\overbrace{\ \ \ \ \ \ \ \ \ \ \ \ }}}}};
	(-8.5,-6)*+{{\color{red} \mathfrak{X}_{v}}};
	(-1,-4.5)*+{\begin{tikzpicture}
		\draw[black, thick,->] (0,.25) to [out=0,in=-130] (0,1);
	\end{tikzpicture}};
	(-3,-9)*+{\mbox{{\smaller $x_{1}$}}};
	\end{xy}
	\\[30pt]
	\begin{xy}
	(0,0)*+{\begin{tikzpicture}
		\draw[red, very thick] (0,0) -- (1,0);
		\draw[black, very thick] (1,0) -- (2,0);
		\fill[red] (0,0) circle (.075);
		\fill[black] (2,0) circle (.075);
		\fill[red] (1,0) circle (.075);
	\end{tikzpicture}};
	(-4.5,3)*+{{\color{red} P_{2}}};
	(-10,-3)*+{{\color{red} v}};
	\end{xy}
	&
	\ \ \ \ \ \ \ \mbox{\larger\larger  $\xmapsto{\ \ \ }$}\ \ \ \ \ \ 
	&
	\begin{xy}
	(0,0)*+{\includegraphics[scale=.75]{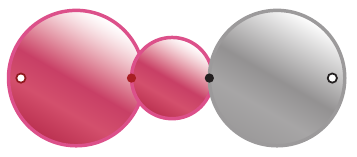}};
	(-3,7.5)*+{\rotatebox{-15}{\color{red} $\overset{\mbox{$\mathfrak{X}_{\!P_{2}}$}}{{\color{pink}\overbrace{\ \ \ \ \ \ \ \ \ \ }}}$}};
	(-11.5,-6)*+{{\color{red} \mathfrak{X}_{v}}};
	(-3.5,-4.5)*+{\begin{tikzpicture}
		\draw[black, thick,->] (0,.25) to [out=0,in=-130] (0,1);
	\end{tikzpicture}};
	(-5.5,-9)*+{\mbox{{\smaller $x_{2}$}}};
	\end{xy}
	\\[30pt]
	\begin{xy}
	(0,0)*+{\begin{tikzpicture}
		\draw[red, very thick] (0,0) -- (.5,0);
		\draw[black, very thick] (.5,0) -- (2,0);
		\fill[red] (0,0) circle (.075);
		\fill[black] (2,0) circle (.075);
		\fill[black] (1,0) circle (.075);
		\fill[red] (.5,0) circle (.075);
	\end{tikzpicture}};
	(-7.25,3)*+{{\color{red} P_{3}}};
	(-10,-3)*+{{\color{red} v}};
	\end{xy}
	&
	\ \ \ \ \ \ \ \mbox{\larger\larger  $\xmapsto{\ \ \ }$}\ \ \ \ \ \ 
	&
	\begin{xy}
	(0,0)*+{\includegraphics[scale=.75]{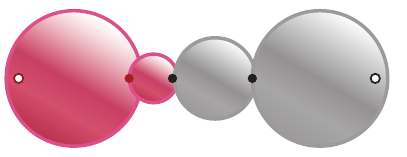}};
	(-5,7)*+{\rotatebox{-25}{\color{red} $\overset{\mbox{$\mathfrak{X}_{\!P_{3}}$}}{{\color{pink}\overbrace{\ \ \ \ \ \ \ }}}$}};
	(-13.5,-6)*+{{\color{red} \mathfrak{X}_{v}}};
	(-6,-4.5)*+{\begin{tikzpicture}
		\draw[black, thick,->] (0,.25) to [out=0,in=-130] (0,1);
	\end{tikzpicture}};
	(-8,-9)*+{\mbox{{\smaller $x_{3}$}}};
	\end{xy}
	\\[30pt]
	\mbox{{\larger\larger $\vdots$}}\ \ \ \ \ \ \ \ \ \!
	&
	\mbox{{\larger\larger $\vdots$}}
	&
	\ \ \ \ \ \ \!\mbox{{\larger\larger $\vdots$}}
	\end{array}
	$}
	$$
	\caption{The formal $R$-schemes associated to several decompositions of the interval $P_{1}$ inside $N_{\mathbb{R}}=\mathbb{R}$. For each decomposition, the $k$-variety underlying the $R$-scheme $\mathfrak{X}_{v}$ associated to the left endpoint $v$ is a copy of $\mathbb{A}^{1}_{k}\!-\!\{0\}$ that does not change from decomposition to decomposition. The nodal point $x_{i}$ that partially compactifies $\mathbb{A}^{1}_{k}\!-\!\{0\}$ in each of the $R$-schemes ``stays fixed" as further decompositions move every $\mathbb{P}^{1}_{\!k}$ component further and further from $\mathfrak{X}_{v}$.}
	\label{bubbling off}
	\end{figure}
	
Clearly $\mathfrak{X}_{v}$ sits in the inverse limit of $R$-models associated to all subdivisions of the interval, but the inverse limit also contains a point in the closure of $\mathfrak{X}_{v}$, which does {\em not} map to the component of any model associated to a segment {\em not} containing $v$. The sequence of points $x_{i}$ in Figure \ref{bubbling off} form an inverse system, and hence determine a single point $x$ in the adic tropicalization $\mathrm{Ad}(\mathbb{P}^{1})$.

\vskip .75cm


\section{The adic limit Theorem}\label{Section-Main Theorem}

	In this final \S\ref{Section-Main Theorem}, we explain that adic tropicalizations satisfy a form of ``analytification is the limit of all tropicalizations," namely, that one can recover the Huber analytification $X^{\mathrm{ad}}$ of any closed subvariety of a proper toric $K$-variety, along with the sheaf of power bound sections $\mathscr{O}^{\circ}_{\!X^{\mathrm{ad}}}$ on $X^{\mathrm{ad}}$, from an inverse limit of adic tropicalizations of $X$. The exact statement appears in Theorem \ref{Huberlimit} and Proposition \ref{Nice system, man!} below. We begin with a brief review of the corresponding theorem in the setting of Berkovich analytic spaces, as proved in \cite{P09} and \cite{FGP}.
	

\subsection{The limit theorem in the Berkovich setting}\label{systems}\ 

Since each multiplicative seminorm $K[U_{\sigma}]\lra\mathbb{R}_{\geq0}$ induces a homomorphism $S_{\sigma}\lra\mathbb{R}\sqcup\{\infty\}$ of semigroups,
each closed embedding
$
\textit{\i}:X\xymatrix{{}\ar@{^{(}->}[r]&{}}Y_{\Sigma}
$
comes with a continuous map of topological spaces
$$
\pi_{\textit{\i}}:\an{X}\lra\Trop{X}{\textit{\i}},
$$
called the {\em tropicalization map} associated to the closed embedding $\textit{\i}$.

If $f_{\phi}:Y_{\Sigma}\lra Y_{\Sigma'}$ is a toric morphism, necessarily induced by morphisms $\phi:\Sigma\lra\Sigma'$ of fans, and if
$
\textit{\i}:X\xymatrix{{}\ar@{^{(}->}[r]&{}}Y_{\Sigma}
$
and
$
\textit{\i}':X\xymatrix{{}\ar@{^{(}->}[r]&{}}Y_{\Sigma'}
$
are closed embeddings for which the diagram
\begin{equation}\label{systemmorphism}
\begin{aligned}
\begin{xy}
(0,0)*+{X}="1";
(15,7)*+{Y_{\Sigma}}="2";
(15,-7)*+{Y_{\Sigma'}}="3";
{\ar@{^{(}->}^{\textit{\i}} "1"; "2"};
{\ar@{^{(}->}_{\textit{\i}'} "1"; "3"};
{\ar@{->}^{f_{\phi}} "2"; "3"};
\end{xy}
\end{aligned}
\end{equation}
commutes, then the corresponding tropicalization maps $\pi_{\!\textit{\i}}$ and $\pi_{\textit{\i}'}$ fit into a commutative diagram
\begin{equation}\label{tropmorphism}
\begin{aligned}
\ \ \ \ \ \ 
\begin{xy}
(0,0)*+{\an{X}\!\!}="1";
(20,10)*+{\Trop{X}{\textit{\i}}}="2";
(20,-10)*+{\Trop{X}{\textit{\i}'}}="3";
{\ar@{->}^{\pi_{\!\textit{\i}}} "1"; "2"};
{\ar@{->}_{\pi_{\!\textit{\i}'}} "1"; "3"};
{\ar@{->}^{\mathrm{Trop}(f_{\phi})} "2"; "3"};
\end{xy}
\end{aligned}
\end{equation}

Let $\mathcal{S}$ denote any small category whose set of objects consists of some family of closed embeddings $\textit{\i}:X\xymatrix{{}\ar@{^{(}->}[r]&{}}Y_{\Sigma}$,
and in which each homset $\mathrm{Hom}_{\mathcal{S}}(\textit{\i},\textit{\i}')$ consists of all commutative diagrams of the above form (\ref{systemmorphism}). Then the system of tropicalization maps $\pi_{\textit{\i}}$ associated objects and morphisms in $\mathcal{S}$ induces a map
\begin{equation}\label{Berkcompmap}
\pi\ \overset{\mathrm{def}}{=}\ \varprojlim_{\mathcal{S}}\pi_{\textit{\i}}\ \!:\ \an{X}\lra\varprojlim_{\mathcal{S}}\Trop{X}{\textit{\i}}
\end{equation}
We call this map the {\em Berkovich comparison map}.

In \cite{P09}, S. Payne showed that if $X$ admits at least one closed embedding into a quasi-projective toric variety $Y_{\Sigma}$, and if $\mathcal{S}$ is the system of all closed embeddings into quasiprojective toric varieties, then the Berkovich comparison map (\ref{Berkcompmap}) is a homeomorphism.
In \cite{FGP}, P. Gross, S. Payne, and the present author extended this result to the case of any $K$-scheme admitting at least one closed embedding into a toric variety. This result is a consequence of the following Theorem \ref{Berkovich}, which describes conditions under which an arbitrary system $\mathcal{S}$ of closed toric embeddings gives rise to a homeomorphism (\ref{Berkcompmap}):

\vskip .4cm

\begin{theorem}\label{Berkovich}
{\bf \cite[Theorem 1.1]{FGP}}\ \ 
If the system $\mathcal{S}$ satisfies conditions {\bf (C1)} and {\bf (C2)} below, then the Berkovich comparison map (\ref{Berkcompmap}) is a homeomorphism:
\vskip .3cm
\begin{itemize}
\item[{\bf (C1)}]
If $\textit{\i}:X\xymatrix{{}\ar@{^{(}->}[r]&{}}Y_{\Sigma}$ and $\textit{\i}':X\xymatrix{{}\ar@{^{(}->}[r]&{}}Y_{\Sigma'}$ are closed embeddings in $\mathcal{S}$, then their product
$
\textit{\i}\times\textit{\i}':X\xymatrix{{}\ar@{^{(}->}[r]&{}}Y_{\Sigma\times\Sigma'}
$
is also in $\mathcal{S}$;
\item[{\bf (C2)}]
\vskip .2cm
There exists a finite affine open cover $\big\{\!\xymatrix{U_{i}\ar@{^{(}->}[r]&X}\!\big\}$ such that for each nonzero regular function $f\in K[U_{i}]$, there is some closed embedding
$
\textit{\i}:X\xymatrix{{}\ar@{^{(}->}[r]&{}}Y_{\Sigma}
$
in $\mathcal{S}$ that realizes $U_{i}$ as the preimage of a torus-invariant open affine, and realizes $f$ as the pullback of a monomial.
\end{itemize}
\end{theorem}

\vskip .3cm


\subsection{Adic tropicalization morphisms}\label{Adic tropicalization morphisms}\ 

Fix a $K$-variety $X$, and let 
$
\xymatrix{
\textit{\i}:X
\ar@{^{(}->}[r]
&
Y_{\Sigma}
}
$
be a closed embedding into a proper toric variety $Y_{\Sigma}$. As explained in \S\ref{subsection: formal Gubler models}, each $\Gamma$-admissible fan $\Delta$ in $N_{\mathbb{R}}\times\mathbb{R}_{\ge0}$, satisfying
	\begin{equation}\label{inset: two conditions}
	\mathrm{supp}(\Delta)\ =\ N_{\RR}\times\RR_{\ge0}
	\ \ \ \ \ \ \ \ \mbox{and}\ \ \ \ \ \ \ \ 
	\Delta\cap\big(N_{\RR}\times\{0\}\big)
	\ \ =\ \ \Sigma\times\{0\},
	\end{equation}
gives rise to an admissible formal model $\mathfrak{X}_{(\textit{\i},\Delta)}$ of $X$. By \ref{specializations of adic spaces}, this model $\mathfrak{X}_{(\textit{\i},\Delta)}$ comes with a morphism of locally topologically ringed spaces
	\begin{equation}\label{inset: second appearance of specialization}
\mathrm{sp}_{\mathfrak{X}_{(\textit{\i},\Delta)}}:\big(\ad{X},\mathscr{O}^{\ \!\!\circ}_{\ad{X}}\big)\lra\big(\mathfrak{X}_{(\textit{\i},\Delta)},\mathscr{O}_{\mathfrak{X}_{(\textit{\i},\Delta)}}\big).
	\end{equation}
Our very construction of the adic tropicalization $\mathrm{Ad}(X,\textit{\i})$ as an inverse limit implies that as $\Delta$ ranges over all $\Gamma$-admissible fans in $N_{\RR}\times\RR_{\ge0}$ satisfying the twin conditions (\ref{inset: two conditions}), the resulting specialization morphisms (\ref{inset: second appearance of specialization}) give rise to a single morphism of locally topologically ringed spaces
	\begin{equation}\label{adic tropicalization morphism}
	\varpi_{\textit{\i}}
	\ \ \overset{\mathrm{def}}{=}
	\varprojlim_{\ \bold{REC}_{\Sigma,\textit{\i}}}\mathrm{sp}_{(\textit{\i},\Delta)}\ \ :\ \ (\ad{X},\mathscr{O}^{\ \!\!\circ}_{\ad{X}})\lra\big(\mathrm{Ad}(X,\textit{\i}),\ \!\mathscr{O}_{\mathrm{Ad}(X,\textit{\i})}\big).
	\end{equation}
This morphism (\ref{adic tropicalization morphism}) is the adic counterpart to the tropicalization map
	$$
	\pi_{\textit{\i}}:\an{X}\lra\mathrm{Trop}(X,\textit{\i}),
	$$
and so we refer to (\ref{adic tropicalization morphism}) as the {\em adic tropicalization morphism} at $\textit{\i}$.

\begin{definition} A {\em system of embeddings} for $X$ is any category $\mathcal{S}$ such that $\mathrm{ob}_{\ \!}\mathcal{S}$ is some class of closed embeddings $\textit{\i}:X\mono Y_{\Sigma}$ into proper toric varieties, and where the set of morphisms between two such embeddings is the set of all commutative diagrams
	\begin{equation}\label{system 2}
	\begin{aligned}
	\begin{xy}
	(0,0)*+{X}="1";
	(15,7)*+{Y_{\Sigma}}="2";
	(15,-7)*+{Y_{\Sigma'}}="3";
	{\ar@{^{(}->}^{\textit{\i}} "1"; "2"};
	{\ar@{^{(}->}_{\textit{\i}'} "1"; "3"};
	{\ar@{->}^{f_{\phi}} "2"; "3"};
	\end{xy}
	\end{aligned}
	\end{equation}
where $f_{\phi}$ is any torus-equivariant morphism induced by a morphism $\phi:\Sigma\lra\Sigma'$ of fans.
\end{definition}

\begin{remark}
If $\Sigma$ and $\Sigma'$ are complete fans in $N_{\mathbb{R}}$ and $N'_{\mathbb{R}}$, then for each fan $\Delta'$ in $N'_{\mathbb{R}}\times\mathbb{R}_{\ge0}$ satisfying the twin conditions (\ref{inset: two conditions}), we can refine any fan $\Delta$ in $N_{\mathbb{R}}\times\mathbb{R}_{\ge0}$ to a fan $\Delta_{+}$ such that the product map
	$$
	f_{\phi}\times\mathrm{id}:N_{\mathbb{R}}\times\mathbb{R}_{\ge0}\ \lra\  N'_{\mathbb{R}}\times\mathbb{R}_{\ge0}
	$$
induces a morphism of fans $f_{\phi}\times\mathrm{id}:\Delta_{+}\lra\Delta'$. Thus if $\mathcal{S}$ is a system of embeddings for $X$, then each morphism (\ref{system 2}) induces a morphism of locally topologically ringed spaces
	\begin{equation}\label{inset: induced adic trop map}
	\mathrm{Ad}(X,f_{\phi}):\mathrm{Ad}(X,\textit{\i})\lra\mathrm{Ad}(X,\textit{\i}').
	\end{equation}
These morphisms (\ref{inset: induced adic trop map}) give the assignment $\big(\textit{\i}:X\mono Y_{\Sigma}\big)\longmapsto\mathrm{Ad}(X,\textit{\i})$ the structure of a functor
	$$
	\mathrm{Ad}(X,-):\mathcal{S}\lra\bold{LTRS}.
	$$
The inverse limit of the adic tropicalization morphisms (\ref{adic tropicalization morphism}) over $\mathcal{S}$ provides us with a morphism of locally topologically ringed spaces
	\begin{equation}\label{inset: comparison map}
	\varprojlim_{\textit{\i}\in\mathcal{S}}\varpi_{\textit{\i}}\ \!:\ (\ad{X},\mathscr{O}^{\ \!\!\circ}_{\ad{X}})\ \!\xrightarrow{\ \ \ \ \ }\ \!\varprojlim_{\textit{\i}\in\mathcal{S}}\big(\mathrm{Ad}(X,\textit{\i}),\ \!\mathscr{O}_{\mathrm{Ad}(X,\textit{\i})}\big),
	\end{equation}
an adic counterpart to the comparison map (\ref{Berkcompmap}) that appears in the context of Berkovich analytic spaces.
\end{remark}

\vskip .3cm

\begin{theorem}\label{Huberlimit}
Let $X$ be a proper variety over an algebraically closed, non-trivially valued non-Arhimedean field $K$, and let $\mathcal{S}$ be a system of embeddings for $X$. If $\mathcal{S}$ satisfies conditions {\bf (C0)}, {\bf (C1)}, and {\bf (C2)} below, then the adic comparison morphism (\ref{inset: comparison map}) is an isomorphism of locally topologically ringed spaces.
\begin{itemize}
\item[{\bf (C0)}]\vskip .25cm
If $\textit{\i}:X\xymatrix{{}\ar@{^{(}->}[r]&{}}Y_{\Sigma}$ is a closed embedding in $\mathcal{S}$, with $\mathbb{T}$ the dense torus in $Y_{\Sigma}$, then any translation of $\textit{\i}$ by a $K$-point of $\mathbb{T}$ is again in $\mathcal{S}$;
\item[{\bf (C1)}]\vskip .25cm
If $\textit{\i}:X\xymatrix{{}\ar@{^{(}->}[r]&{}}Y_{\Sigma}$ and $\textit{\i}':X\xymatrix{{}\ar@{^{(}->}[r]&{}}Y_{\Sigma'}$ are closed embeddings in $\mathcal{S}$, then their product
$
\textit{\i}\times\textit{\i}':X\xymatrix{{}\ar@{^{(}->}[r]&{}}Y_{\Sigma\times\Sigma'}
$
is also in $\mathcal{S}$;
\item[{\bf (C2)}]\vskip .25cm
There exists a finite affine open cover $\big\{\!\xymatrix{U_{i}\ar@{^{(}->}[r]&X}\!\big\}$ such that for each nonzero regular function $f\in K[U_{i}]$, there is a closed embedding
$
\textit{\i}:X\xymatrix{{}\ar@{^{(}->}[r]&{}}Y_{\Sigma}
$
in $\mathcal{S}$ that realizes $U_{i}$ as the preimage of a torus-invariant open affine, and realizes $f$ as the pullback of a monomial.
\end{itemize}
\end{theorem}

\vskip .3cm

\begin{corollary}
{\bf (Adic limit theorem on point sets).}
If $\mathcal{S}$ satisfies conditions {\bf (C0)} through {\bf (C2)}, then the point set underlying the Huber analytification of $X$ is in natural bijection with the inverse limit of all extended exploded tropicalizations of $X$:
	$$
	|X^{\mathrm{ad}}|
	\ \ \cong\ \ 
	\varprojlim_{\textit{\i}\in\mathcal{S}}\ \mathfrak{Trop}(X,\textit{\i})
	$$
\end{corollary}
\begin{proof}
This follows immediately from Theorems \ref{theorem: topology on exploded trop} and \ref{Huberlimit}.
\end{proof}

\vskip .3cm

\begin{remark}
One proves Theorem \ref{Huberlimit} by showing that the system of all Gubler models of $X$, for all closed embeddings $\textit{\i}:X\mono Y_{\Sigma}$ in proper toric embeddings, is cofinal in the system of all admissible formal $R$-models of $X$. The key technical step in the proof is Proposition \ref{majorizing} below. As in the statement of Theorem \ref{Huberlimit}, let $X$ be a proper $K$-variety over an algebraically closed, non-trivially valued non-Arhimedean field $K$, and let $\mathcal{S}$ be a system of embeddings for $X$ satisfying conditions {\bf (C0)},  {\bf (C1)}, and  {\bf (C2)} of Theorem \ref{Huberlimit}.
\end{remark}

\vskip .3cm

\begin{proposition}\label{majorizing}
Every admissible formal $R$-model $\mathfrak{X}$ of $X$ is dominated by a formal Gubler model. More precisely, if $\mathfrak{X}$ is any admissible formal $R$-model of $X$, then there exists
	\begin{itemize}
	\item[{\bf (i)}]\vskip .25cm
	a closed embedding $\textit{\i}:X\mono Y_{\Sigma}$ into the proper toric variety associated to some complete fan $\Sigma$ in some $\RR$-vector space $N_{\RR}$,
	\item[{\bf (i)}]\vskip .25cm
	a $\Gamma$-admissible fan $\Delta$ in $N_{\RR}\times\RR_{\ge0}$ satisfying the twin conditions (\ref{inset: two conditions}),
	\end{itemize}
	\vskip .25cm
such that the formal Gubler model $\mathfrak{X}_{(\textit{\i},\Delta)}$ associated to the pair $(\textit{\i},\Delta)$ admits a proper morphism
	$$
	\mathfrak{X}_{(\textit{\i},\Delta)}
	\ \lra\!\!\!\!\rightarrow\ 
	\mathfrak{X}
	$$
of formal $R$-models of $X$.
\end{proposition}

\vskip .3cm

To close, we give several examples of systems of embeddings that satisfy the hypotheses of Theorem \ref{Huberlimit}, so that the comparison morphism (\ref{inset: comparison map}) becomes an isomorphism.

\vskip .3cm

\begin{proposition}\label{Nice system, man!}
Let $X$ be a $K$-variety satisfying the two-point condition {\bf (A2)} appearing in \S\ref{Adic tropicalization of a closed embedding} above. Then in each of the following situations, the system $\mathcal{S}$ satisfies conditions {\bf (C0)}, {\bf (C1)}, and {\bf (C2)} of Theorem \ref{Huberlimit}:
\vskip .3cm
\begin{itemize}
\item[{\bf (i)}]
$X$ proper, $\mathcal{S}$ the category of all closed embeddings of $X$ into proper toric varieties;
\vskip .2cm
\item[{\bf (ii)}]
$X$ projective, $\mathcal{S}$ the category of all closed embeddings of $X$ into projective toric varieties;
\vskip .2cm
\item[{\bf (iii)}]
$X$ smooth and proper, $\mathcal{S}$ the category of all closed embeddings of $X$ into smooth, proper toric varieties.
\end{itemize}
\end{proposition}
\begin{proof}
Case {\bf (ii)} follows from \cite[Lemma 4.3]{P09}. Cases {\bf (i)} and {\bf (iii)} follow from the construction described in the proof of \cite[Theorem 4.2]{FGP}, using the fact that we can construct the toric embedding with the requisite property in each of these cases, as explained in \cite[Theorem A]{Wlod}.
\end{proof}

\vskip .75cm


\bibliographystyle{acm}
\bibliography{Intro_to_adic_trop}
\vskip .75cm

\end{document}